\documentclass[12pt,oneside,english]{amsart}

\usepackage[T1]{fontenc}
\usepackage[latin9]{inputenc}
\usepackage{geometry}
\usepackage{babel}
\usepackage{amsbsy}
\usepackage{amsthm}
\usepackage{amsmath}
\usepackage{amsfonts}
\usepackage[unicode=true,pdfusetitle,
 bookmarks=true,bookmarksnumbered=false,bookmarksopen=false,
 breaklinks=false,pdfborder={0 0 0},backref=false,colorlinks=false]
 {hyperref}

\makeatletter
\numberwithin{equation}{section}
\numberwithin{figure}{section}
\theoremstyle{plain}
\newtheorem{thm}{\protect\theoremname}[section]
\newtheorem{prop}[thm]{\protect\propositionname}
\newtheorem{lem}[thm]{\protect\lemmaname}

\theoremstyle{definition}
\newtheorem{defn}[thm]{\protect\definitionname}
\newtheorem{rem}[thm]{\protect\remarkname}  
\newtheorem{example}[thm]{\protect\examplename}

\usepackage{color}
\usepackage{graphicx}

\makeatother

  \providecommand{\lemmaname}{Lemma}
\providecommand{\theoremname}{Theorem}
\providecommand{\definitionname}{Definition}
  \providecommand{\remarkname}{Remark}
  \providecommand{\examplename}{Example}
   \providecommand{\propositionname}{Proposition}
  
\begin{document}

\title{Surfaces and hypersurfaces as the joint spectrum of matrices}

\author{Patrick H. DeBonis}

\address{Department of Mathematics and Statistics, University of New Mexico,
Albuquerque, New Mexico 87131, USA}

\curraddr{Department of Mathematics, Purdue University
150 N. University Street, West Lafayette, Indiana 47907, USA}

\author{Terry A. Loring}

\address{Department of Mathematics and Statistics, University of New Mexico,
Albuquerque, New Mexico 87131, USA}

\author{Roman Sverdlov}

\address{Department of Mathematics and Statistics, University of New Mexico,
Albuquerque, New Mexico 87131, USA}

\subjclass{47A13,46L85, 15A18}

\keywords{Clifford spectrum, joint spectrum, emergent topology, Hermitian matrices}

\begin{abstract}
The Clifford spectrum is an elegant way to define the joint spectrum of several
Hermitian operators.  While it has been know that for examples as small as three
$2$-by-$2$ matrices the Clifford spectrum can be a two-dimensional manifold, 
few concrete examples have been investigated.  
Our main goal is to generate examples of the Clifford spectrum of three or four
matrices where, with the assistance of a computer algebra package, we can
calculate the Clifford spectrum.
\end{abstract}

\maketitle


\tableofcontents{}

\section{Introduction}

The Clifford spectrum is one way extend the concept of joint spectrum
of commuting matrices to work for noncommuting operators. We are only
interested in Hermitian matrices as in the back of our minds we envision
applications to quantum physics and string theory. Given $(X_{1},\dots,X_{d})$,
where the $X_{j}$ are all $n$-by-$n$ Hermitian matrices, we define
a Dirac-type operator
\[
L(X_{1},\dots,X_{d})=\sum X_{j}\otimes\gamma_{j}
\]
where the $\gamma_{j}$ are $d$ matrices that satisfy the Clifford
relations
\begin{equation}
\begin{aligned}
\gamma_{j}^{*} & =\gamma_{j}\quad(\forall j) \\
\gamma_{j}^{2} & =I\quad(\forall j)\\
\gamma_{j}\gamma_{k} & =-\gamma_{k}\gamma_{j}\quad(j\neq k)
\label{eq:gamma_rep}
\end{aligned} .
\end{equation}
We can use $L(X_{1},\dots,X_{d})$ to determine only if $\boldsymbol{0}$
is in the Clifford spectrum. To find the full spectrum, we shift the
matrices by scalars, and define 
\begin{align*}
L_{\boldsymbol{\lambda}}(X_{1},\dots,X_{d}) &=L(X_{1}-\lambda_{1},\dots,X_{d}-\lambda_{d}) \\
&=\sum\left(X_{j}-\lambda_{j}\right)\otimes\gamma_{j}.
\end{align*}
Due to many clashes of terminology between mathematics and physics,
it seems now  prudent, as discussed in \cite{LoringSchuBa_even_AIII}, to
call $L_{\boldsymbol{\lambda}}$
the \emph{spectral localizer} of the $d$-tuple $(X_{1},\dots,X_{d})$.

\begin{defn}
The \emph{Clifford spectrum} of $d$-tuple $(X_{1},\dots,X_{d})$ of Hermitian
matrices is the set of $\lambda$ in $\mathbb{R}^{d}$ such that $L_{\boldsymbol{\lambda}}(X_{1},\dots,X_{d})$
is singular. This is denoted $\Lambda(X_{1},\dots,X_{d})$.
\end{defn}

\begin{rem}
This definition works for Hermitian operators, even when unbounded.
We will focus on the matrix case, except in a few comments and examples.
\end{rem}

It was Kisil \cite{KisilCliffordSpectrum} who noticed that the Clifford spectrum equals
the Taylor spectrum in the case where the $X_{j}$ all commute with
each other. In the case of finite matrices, a singular localizer at 
$\boldsymbol{\lambda}$ implies there is a joint eivenvector with
eigenvalues the components of $\boldsymbol{\lambda}$, and this is
exactly what any form of joint spectrum should mean for commuting
finite matrices.  We will see a more general result in \S \ref{sec:Bound_variance},
where it is shown that for almost commuting matrices we can associate
to points in the Clifford spectrum vectors with small variation with
respect to each $X_j$.

Kisil also used the theory of monogenic functions to prove that
the Clifford spectrum is always nonempty, and indeed compact. 
However, it does not have to be a
finite set when computed for finite matrices that don't commute.

In string theory, the Clifford spectrum is used, but tends to be called
the ``emergent geometry'' \cite{berenstein2012matrix},
 or the ``set of probe points''
\cite{SchneiderbauerMeasuringFiniteGeom} etc. In that context, the Clifford spectrum
consists of all the locations
where a fermionic probe of a D brane can lead to low energy resonance.

For some calculations we will look at the square of the localizer.  It is important to note that the square of this Dirac-type matrix is not exactly the corresponding Laplace-type matrix.
Indeed, one can calculate \cite{LoringPseudospectra} that
\begin{equation}
\left(L_{\boldsymbol{\lambda}}(X_{1},\dots,X_{d})\right)^{2}
=\sum_{j=1}^{d}\left(X_{j}-\lambda_{j}\right)^{2}\otimes I+\sum_{j<k}[X_{j},X_{k}]\otimes\gamma_{j}\gamma_{k}.
\label{eq:square_of_Localizer}
\end{equation}

Why not use directly a Laplace-type operator
to define a spectrum? This will be correct in the commuting case. 

\begin{defn}
The \emph{Laplace spectrum }of Hermitian $d$-tuple $(X_{1},\dots,X_{d})$
is the set of $\lambda$ in $\mathbb{R}^{d}$ such that 
\[
\sum_{j=1}^{d}\left(X_{j}-\lambda_{j}\right)^{2}
\]
is singular.
\end{defn}

The Laplace spectrum is used in string theory \cite{SchneiderbauerMeasuringFiniteGeom}.
We will see it has a flaw that keeps it out of general use. In some cases,
when the commutators are small, one might be able to prove that the Laplace spectrum is
a decent approximation of the Clifford spectrum.

An issue with the Clifford spectrum is that it is very hard to
work examples by hand. Looking hard at the math and string theory
literature, we find a only a handful of explicit examples where the
Clifford spectrum is known. Indeed, Schneiderbauer and
Steinacker \cite{SchneiderbauerMeasuringFiniteGeom}, and also Sykora \cite{sykora2016fuzzy},
use a computer algebra package for many fuzzy geometry calculations. We are taking
on a similar challenge, using a computer algebra package to find more
examples.

We will primarily use a generalized characteristic polynomial to calculate
the Clifford spectrum of various examples. The generalized characteristic
polynomial probably first appeared in work by Berenstein, Dzienkowski
and Lashof-Regas \cite{berenstein2015spinning}.

\begin{defn}
The \emph{characteristic polynomial} of the $d$-tuple $(X_{1},\dots,X_{d})$
is the polynomial, in real variables $\lambda_{1}\dots,\lambda_{d}$,
\[
\boldsymbol{\lambda}\mapsto\det(L_{\boldsymbol{\lambda}}(X_{1},\dots,X_{d}))
\]
which we denote $\mathrm{char}(X_{1},\dots,X_{d})$.
\end{defn}

The equation $\mathrm{char}(X_{1},\dots,X_{d})=0$ determines
the Clifford spectrum. This can become
a polynomial with many monomials in many variable even in rather modest
examples. Hence the need for a computer assist and an experimental
approach.

Some of the complexity from increasing $d$, the number of matrices, comes from
the fact that the $\gamma_j$ get bigger.  It is best to use an irreducible representation
of (\ref{eq:gamma_rep}), which means that each $\gamma$ is $g$-by-$g$ for
\begin{equation*}
g= 2^{\lfloor d/2\rfloor}
\end{equation*}
as one can see from  \cite{OkuboCliffordReps}, for example.
The wrong value for
$g$ was used in \cite[\S 1]{LoringPseudospectra} and so the estimates there were not correct as stated.  See Section~\ref{sec:Bound_variance}.

Section~\ref{sec:Bound_variance} discusses the variance of joint approximate eigenvalues.
Section~\ref{sec:One-or-two} discusses the cases of one or two matrices (or operators)
where the Clifford spectrum agrees with the ordinary single-operator spectrum. 
Section~\ref{sec:Three-Hermitian-matrices} looks at the case of three matrices,
where the Clifford spectrum
can be a surface. This is where we have the most examples, as surfaces
in three space are easy to display.  
Section~\ref{sec:Four-Hermitian-matrices} looks as the case of
four matrices, where the calculations and visualization become harder.
Section~\ref{SymmetryClasses} looks are variations
on the localizer and index that assist with plotting and proving the
stability of the Clifford spectrum. 
Many of these examples in Section~\ref{sec:Three-Hermitian-matrices}
and the discussion of the archetypal polynomial are from
the thesis of DeBonis \cite{Debonis2019}.

We will use \emph{mathematical notation} throughout.  Most importantly, Hermitian matrices are those for which $X^* = X$, and so complex conjugation is indicated by $X^*$.  In several places we will focus on unit vectors, so have in mind states of a quantum system.  Since the word state means something different in operator algebras, for this we stick the the neutral terminology.

The convention we prefer for identifying a tensor product of matrices with a
larger matrix is the one such that
\begin{equation*}
A\otimes \begin{pmatrix} a& b\\ c & d \end{pmatrix}
=
\begin{pmatrix} aA& bA\\ cA & dA \end{pmatrix}
\end{equation*}
and this is \emph{opposite} of the convention used by the \texttt{KroneckerProduct} operation in
Mathematica.  

\section{Bounds on variance} \label{sec:Bound_variance}

Suppose $\mathbf{v}$ is a unit vector and $X$ is a Hermitian matrix.  Two
important quantities when considering quantum measurement are the expection
value of $X$ with respect to $\mathbf{v}$
\begin{eqnarray*}
\textnormal{E}(X)_\mathbf{v} = \langle X\mathbf{v}, \mathbf{v} \rangle
\end{eqnarray*} 
and the variance of $X$ with respect to $\mathbf{v}$
\begin{eqnarray*}
\textnormal{Var}(X)_\mathbf{v} =
 \langle X^2\mathbf{v}, \mathbf{v} \rangle - \langle X\mathbf{v}, \mathbf{v} \rangle ^2 .
\end{eqnarray*} 
For any scalar $\lambda$ we have
\begin{equation*}
\left\langle (X-\lambda)^{2}\mathbf{v},\mathbf{v}\right\rangle =\left\langle X^{2}\mathbf{v},\mathbf{v}\right\rangle -2\lambda\left\langle X\mathbf{v},\mathbf{v}\right\rangle +\lambda^{2}
\end{equation*}
and 
\begin{equation*}
\left\langle (X-\lambda)\mathbf{v},\mathbf{v}\right\rangle ^{2}
=\left\langle X\mathbf{v},\mathbf{v}\right\rangle ^{2}-2\lambda\left\langle X\mathbf{v},\mathbf{v}\right\rangle +\lambda^{2}
\end{equation*}
so we see that 
\begin{equation}
\textnormal{Var}(X - \lambda)_\mathbf{v} = \textnormal{Var}(X)_\mathbf{v}.
\label{eqn:how_Var_shifts}
\end{equation}
On the other hand,
\begin{equation}
\textnormal{E}(X - \lambda)_\mathbf{v} = \textnormal{E}(X)_\mathbf{v} - \lambda .
\label{eqn:how_E_shifts}
\end{equation}
If the $\textnormal{Var}(X)_\mathbf{v}=0$ then $\mathbf{v}$ is an eigenvector for $X$ for eigenvalue $\textnormal{E}(X)_\mathbf{v}$.

When attempting joint measurement, for observables $X_{1},\dots,X_{d}$, one confronts 
often the impossibility of finding any  unit vector $\mathbf{v}$ that is
simultaneously an eigenvector for all the observables.  There are many 
lower bounds on the variances that make this more precise, such as the
Robertson--Schr\"odinger relation bounding the product of the variance of two observables.
A more recent example of such a lower
bound, due to Chen and Fei \cite{Chen2015sumUncertainty}, gives lower bounds on the
sum of $d$ variances.

We look here at upper bounds on the sum of variances.  Specifically, we
will derive an estimate on how small we can make the variances for if
we choose certain unit vectors that are related to points in the Clifford
spectrum.

\begin{lem}
\label{lem:basic_estimate_on_Xj_and_w}
Suppose $X_{1},\dots,X_{d}$ are Hermitian, $n$-by-$n$ matrices
and $\boldsymbol{\lambda}$ is in $\Lambda(X_{1},\dots,X_{d})$. Then
there is a unit vector $\mathbf{w}$ in $\mathbb{R}^{n}$ such
that 
\begin{equation*}
\sum \left\langle \left(X_{j}-\lambda_{j}\right)^{2}\mathbf{w},\mathbf{w}\right\rangle \leq g \sum_{j<k}\left\Vert \left[X_{j},X_{k}\right]\right\Vert 
\end{equation*}
for $g=2^{\lfloor d/2\rfloor}$.
\end{lem}

\begin{proof}
Since shifting the $X_{j}$ by $\lambda_{j}$ has no effect on the
commutators, we can reduce to the case of $\boldsymbol{\lambda}=\boldsymbol{0}$.
Assume then that $\bf{0}$ is in the Clifford spectrum of
$X_{1},\dots,X_{d}$.  Then there is a vector $\mathbf{z}$  in $\mathbb{R}^{gn}$ 
such that 
\begin{equation}  
L_{\bf{0}} (X_1, \cdots, X_n)\mathbf{z}   = 0
\label{eqn:Lz_assumed_zero}
\end{equation}
One might be tempted to diagonalize $L_{\bf{0}} (X_1, \cdots, X_d)$ so that $\boldsymbol{z}$ can be written down as a column vector with only one single non-zero entry. This, however, would not be the best move: if we change coordinate system, then $X_1 \otimes \gamma_1 + \cdots + X_n \otimes \gamma_d$ would no longer be written in a block form and, therefore, we would no longer be able to isolate $X_j$ and use some of its properties. Therefore, we refrain from diagonalizing and write $\mathbf{z}$ as 
\begin{equation}
\mathbf{z}=\begin{bmatrix}
\mathbf{z}_1 \\
\vdots \\
\mathbf{z}_g \end{bmatrix}
\label{z} 
\end{equation} 
where $\mathbf{z}_k \in \mathbb{R}^n$ for all $k \in \{1, \cdots, g \}$.
From (\ref{eqn:Lz_assumed_zero}) we obtain 
$\left(L_{\bf{0}} (X_1, \cdots, X_n)\right)^2\mathbf{z}   = 0$.  Now (\ref{eq:square_of_Localizer}) tells us 
\begin{equation}
\sum_j (X_j^2 \otimes I_g) \mathbf{z} = - \sum_{j<k} ([X_j, X_k] \otimes (\gamma_j \gamma_k)) \mathbf{z} \nonumber 
\end{equation}
and therefore
\begin{equation}
\left\| \sum_j X_j^2  \mathbf{z}_r \right\| \leq  \sum_{j<k} \|[X_j, X_k]\|  \nonumber 
\end{equation}
for every $r$.
Now we select $r$ in such a way that it maximizes $\| \mathbf{z}_r \|$ and set
\begin{equation*} 
\mathbf{w} =  \frac{1}{\|\mathbf{z}_r\|} \mathbf{z}_r .
\end{equation*}
Thus, 
\begin{equation} 
1 = \| \mathbf{z} \|^2 = \sum_{j=1}^g \| \boldsymbol{z}_j \|^2 \leq g \| \mathbf{z}_r \|^2 
\nonumber 
\end{equation}
and, therefore, $\| \mathbf{z}_r \| \geq 1 / \sqrt{g}$. 
We can now perform the following calculation:
\begin{align*} 
\sum  \left \langle  X_j^2 \mathbf{w}, \mathbf{w} \right\rangle 
& = \left \langle \sum  X_j^2 \mathbf{w}, \mathbf{w} \right\rangle \\
& \leq g \left \langle \sum  X_j^2 \mathbf{z}_r, \mathbf{z}_r \right\rangle \\
&  \leq g \sum_{j<k} \|[X_j, X_k]\|.
\end{align*}
\end{proof}

\begin{thm}
Suppose $X_{1},\dots,X_{d}$ are Hermitian, $n$-by-$n$ matrices
and $\boldsymbol{\lambda}$ is in $\Lambda(X_{1},\dots,X_{d})$. Then
there is a unit vector $\mathbf{w}$ in $\mathbb{R}^{n}$ such
that 
\begin{equation*}
\sum_{j=1}^{d}\textnormal{Var}(X_{j})_{\mathbf{w}}
+ \left|\textnormal{E}(X_{j})_{\mathbf{w}}-\lambda_{j}\right|^{2}
\leq g\sum_{j<k}\left\Vert \left[X_{j},X_{k}\right]\right\Vert 
\end{equation*}
for $g=2^{\lfloor d/2\rfloor}$.
\end{thm}

\begin{proof}
By (\ref{eqn:how_Var_shifts}) and (\ref{eqn:how_E_shifts}) we can again assume, without loss of generality, that $\boldsymbol{\lambda} = \mathbf{0}$.
By Lemma~\ref{lem:basic_estimate_on_Xj_and_w} there exists a unit
vector $\mathbf{w}$ such that 
\begin{equation*}
\sum \left\langle  X_{j} ^{2}\mathbf{w},\mathbf{w}\right\rangle \leq g \sum_{j<k}\left\Vert \left[X_{j},X_{k}\right]\right\Vert .
\end{equation*}
For any Hermitian matrix $X$ and unit vector $\mathbf{v}$ we have
\begin{equation*}
\left\langle  X^2 \mathbf{v},\mathbf{v}\right\rangle
= \textnormal{Var}(X)_\mathbf{v} + \left(\textnormal{E}(X)_\mathbf{v}\right)^2
\end{equation*}
so in this special case we have
\begin{equation*}
\sum \left(\textnormal{Var}(X_j)_\mathbf{w} + \left(\textnormal{E}(X_j)_\mathbf{w}\right)^2\right) \leq g \sum_{j<k} \left\Vert \left[X_{j},X_{k}\right]\right\Vert .
\end{equation*}
\end{proof}

For larger matrices, it will be difficult to determine the exact location of the
Clifford spectrum. A more practical approach is to find $\mathbf{\lambda}$ that are
in the (Clifford) $\epsilon$-pseudospectrum of $X_{1},\dots,X_{d}$, denoted 
$\Lambda_\epsilon(X_{1},\dots,X_{d})$ as defined in \cite{LoringPseudospectra}.
By definition, $\mathbf{\lambda}$ is in $\Lambda_\epsilon(X_{1},\dots,X_{d})$ whenever
\begin{equation}
\label{eqn:pseusospectrum}
\| (L_{\boldsymbol{\lambda}} (X_1, \cdots, X_n)^{-1} \|^{-1} \leq \epsilon.
\end{equation}

In this paper, we will not use the function
\begin{equation*}
\boldsymbol{\lambda} \mapsto \|(L_{\boldsymbol{\lambda}} (X_1, \cdots, X_n)^{-1} \|^{-1}
\end{equation*}
to estimate the Clifford spectrum.  Notice, however, that (\ref{eqn:pseusospectrum}) is
equivalent to the existence of a unit vector $\mathbf{\lambda}$ such that
\begin{equation}
\label{eqn:Lz_approx_zero}
L_{\boldsymbol{\lambda}} (X_1, \cdots, X_n) \mathbf{z} \leq \epsilon.
\end{equation}
This can be proven easily if one considers a unitary diagonalization of the localizer, which is itself Hermitian.

It is rather easy to compute a unit vector that satisfies (\ref{eqn:Lz_approx_zero})
and such vectors can in interesting, as we now show.

\begin{lem}
\label{lem:second_estimate_on_Xj_and_w}
Suppose $X_{1},\dots,X_{d}$ are Hermitian, $n$-by-$n$ matrices
and there is a vector $\mathbf{z}$ in $\mathbb{R}^{gn}$ such that
\rm{(\ref{eqn:Lz_approx_zero})} holds for some $\epsilon \geq 0$. Then
there is a unit vector $\mathbf{w}$ in $\mathbb{R}^{n}$ such
that 
\begin{equation*}
\sum \left\langle \left(X_{j}-\lambda_{j}\right)^{2}\mathbf{w},\mathbf{w}\right\rangle \leq \epsilon^2 + g \sum_{j<k}\left\Vert \left[X_{j},X_{k}\right]\right\Vert 
\end{equation*}
for $g=2^{\lfloor d/2\rfloor}$.
\end{lem}

\begin{proof}
The proof proceeds essentially the same as the proof of Lemma~\ref{lem:basic_estimate_on_Xj_and_w}.
The first difference is we find that
\begin{equation}
\left\| \sum_j X_j^2  \mathbf{z}_r \right\| \leq \epsilon^2 +  \sum_{j<k} \|[X_j, X_k]\|  \nonumber 
\end{equation}
for every $r$, again with the $ \mathbf{z}_r$ the $g$ components of  $ \mathbf{z}$.

\end{proof}

The following now follows from Lemma~\ref{lem:second_estimate_on_Xj_and_w} by the same
argument as above.  Notice that the method to produce $\mathbf{w}$ from $\mathbf{v}$ is to just select the component of $\mathbf{w}$ that is largest and normalize it.

\begin{thm}
Suppose $X_{1},\dots,X_{d}$ are Hermitian, $n$-by-$n$ matrices
and there is a vector $\mathbf{z}$ in $\mathbb{R}^{gn}$ such that
\rm{(\ref{eqn:Lz_approx_zero})} holds for some $\epsilon \geq 0$. Then
there is a unit vector $\mathbf{w}$ in $\mathbb{R}^{n}$ such
that 
\begin{equation*}
\sum_{j=1}^{d}\textnormal{Var}(X_{j})_{\mathbf{w}}
+ \left|\textnormal{E}(X_{j})_{\mathbf{w}}-\lambda_{j}\right|^{2} \\
\leq \epsilon + g\sum_{j<k}\left\Vert \left[X_{j},X_{k}\right]\right\Vert 
\end{equation*}
for $g=2^{\lfloor d/2\rfloor}$.
\end{thm}

\section{One or two Hermitian matrices} \label{sec:One-or-two}

For one or two Hemitian matrices, the  concept of Clifford spectrum overlaps with the usual concept of spectrum of a matrix.

In the case of a single matrix $X$, we can take as Clifford representation
\begin{equation}
\gamma_1 = 1 \label{eqn:CliffordRepOne}
\end{equation}
which means the localizer is just
\begin{equation*}
L_\lambda = X - \lambda
\end{equation*}
with $\lambda$ a real variable.
Since all the eigenvalues of $X$ are real, this makes no real difference
and so the new characteristic polynomial
$\det(L_\lambda)$
is the usual characteristic polynomial.
Thus $\Lambda(X)$  is just the ordinary spectrum of $X$.

The case of two Hermitian matrices $(X,Y)$ also deviates only in
technical ways from an ordinary spectrum.  We will see right away that
it is essentially the spectrum of $X+iY$.
We can take here for Clifford representation
\begin{equation}
\gamma_1=\begin{bmatrix}
0 & 1\\
1 & 0
\end{bmatrix},
\ 
\gamma_2=\begin{bmatrix}
0 & -i\\
i & 0
\end{bmatrix} .
\label{eq:CliffordRepTwo}
\end{equation}
The localizer then becomes
\begin{equation*}
L_{(r,s)}(X,Y) 
=\begin{bmatrix}
0 & \left((X-r)+i(Y-s)\right)^{*}\\
\left((X-r)+i(Y-s)\right) & 0
\end{bmatrix} 
\end{equation*}
and so 
\[
\left|\det\left(L_{(r,s)}(X,Y)\right)\right| =\left|\det\left((X+iY)-(r+is)\right)\right|^{2}.
\]
If we use a complex variable $z=r+is$ on the right that becomes the
square of the absolute value of the usual characteristic polynomial
of $X+iY$. Therefore
\begin{equation}
\label{eqn:twofoldCliffordVSspectrumXplusiY}
(r,s)\in\Lambda(X,Y) \iff  r+is\in\sigma(X+iY).
\end{equation}

\begin{example}
Consider the two matrices 
\[
X=\begin{bmatrix}
0 & 1\\
1 & 0
\end{bmatrix} ,\quad Y=\begin{bmatrix}
0 & -i\\
i & 0
\end{bmatrix} .
\]
Then
\[
X+iY=\begin{bmatrix}
0 & 2\\
0 & 0
\end{bmatrix} 
\]
which has spectrum $\{0\}$. Thus the Clifford spectrum of $(X,Y)$
is just the set $\{(0,0)\}$. On the other hand, the Laplace spectrum
is the zero set of 
\begin{align*}
\det\left(\begin{bmatrix}
-r & 1\\
1 & -r
\end{bmatrix} ^{2}  
+\begin{bmatrix}
-s & -i\\
i & -s
\end{bmatrix} ^{2}\right) 
 &   =\det\begin{bmatrix}
2+r^{2}+s^{2} & -2r+2is\\
-2r-2is & 2+r^{2}+s^{2}
\end{bmatrix} \\
 &   =4+r^{4}+2r^{2}s^{2}+s^{4}.
\end{align*}
The \emph{Laplace spectrum is the empty set} in this simple example.
Here endeth our interest in the Laplace spectrum.
\end{example}

\begin{prop}
For two Hermitian matrices of size $n$, the Clifford spectrum is
a finite set, with between $1$ and $n$ points as elements.
\end{prop}

\begin{proof}
This follows easily by the equivalence of the Clifford spectrum of
two Hermitian matrices with the ordinary spectrum of a single matrix.
\end{proof}

\begin{prop}
For $d$ commuting Hermitian matrices of size $n$, the Clifford spectrum
is a finite set, with between $1$ and $n$ points as elements.
\end{prop}

\begin{proof}
Now we use the equivalence of the Clifford spectrum of commuting Hermitian
matrices with the ordinary joint spectrum. The appropriate version
of the spectral theorem tells us the joint spectrum is a nonempty
finite set of at most $n$ points.
\end{proof}

The argument leading to the equivalence (\ref{eqn:twofoldCliffordVSspectrumXplusiY})
is valid for Hermitian operators as well.  One example is worth examining.

\begin{example}
Let $P$ and $Q$ be the classical position and momentum
operators on $L^{2}(\mathbb{R})$, so 
\[
Qf(x)=xf(x),\quad Pf(x)=-if'(x).
\]
We will see that joint Clifford spectrum $\Lambda(P,Q)$ is all of
$\mathbb{R}^1$.
This is because of its relation with the spectrum of $P+iQ$. Looking
more closely, let us look for eigenvectors, so $f$ with 
\[
(Q+iP)f=(r+is)f.
\]
(If we look at the whole localizer, we need to solve
\[
\begin{bmatrix}
0 & \left(Q-r\right)-i\left(P-s\right)\\
\left(Q-r\right)+i\left(P-s\right) & 0
\end{bmatrix}
\begin{bmatrix}
g\\
f
\end{bmatrix}
=
\begin{bmatrix}
0\\
0
\end{bmatrix}
\]
which is essentially the same.)
This translates to
\[
f'(x)=\left(\alpha-x\right)f(x)
\]
where $\alpha=r+is$. Then
\[
f(x)=e^{-\frac{1}{2}(x-r)^{2} + isx}.
\]
is a (non-normalized) square-integrable solution to this differential equation for
$\alpha = r+is$.  Such a Gaussian is well known to have limited deviation in 
position and momentum, so the spectral localizer method captures what we would
expect in this example.
\end{example}

The previous example is in some sense the limit  as $n \rightarrow \infty$
of an example we consider in Section~\ref{sec:Four-Hermitian-matrices}.  There
the four Hermitian matrices are the Hermitian and anti-Hermitian parts of the usual clock and shift unitary matrices.  

What physicsts call the clock and shift, mathematicians often call Voiculescu's unitaries.  We
want $U$ to be the cyclic shift and $V$ to be a diagonal unitary with eigenvalues winding around the unit circle, specifically as as follows.
For each $n \geq 2$ we define these two $n$-by-$n$ unitary matrices as
\begin{equation}
\label{eqn:define_U}
U_n= \begin{bmatrix} 0 & & &  & 1 \\ 1 & 0 & &  & \\ & \ddots & \ddots & &  \\  & & 1 & 0  \\  & & & 1 & 0  \end{bmatrix}
\end{equation}
and
\begin{equation}
\label{eqn:define_V}
V =  \begin{bmatrix} e^{2 \pi i/n} & & &  &  \\  &  e^{4 \pi i/n}  & &  & \\ &  & \ddots & & \\  & & & e^{2 \pi i (n-1)/n}  \\  & & & & 1  \end{bmatrix}.
\end{equation}

  Now arguing  heuristically, and from a physics perspective, suppose that space is compactified.  Suppose space has diameter is $L$, and further suppose that it is discretized, with lattice spacing $\epsilon$. If $k$ is the row number, $k \in \{1, \cdots, n \}$, we have 
\begin{equation*} k = \frac{x}{\epsilon} \; , \; n = \frac{L}{\epsilon} \end{equation*} 
Therefore, 
\begin{equation*} U_{k,k-1} \approx 1 - \epsilon \frac{\partial}{\partial x} = 1 -i \epsilon p \end{equation*} 
and
\begin{equation*} V_{kk} = e^{2 \pi i k/n} = e^{x/L} \approx 1+ \frac{x}{L} .
\end{equation*} 
This implies that joint spectrum of $U$ and $V$ would roughly correspond to the joint spectrum of $p$ and $x$, if we will be looking for the eigenvalues that are very large rather than very small. If the size of $U$ and $V$ gets larger and larger, the number of eigenvalues would increase as well, which intuitively explains why in the limit we will get a continuous spectrum.

\section{Three Hermitian matrices }  \label{sec:Three-Hermitian-matrices}

In the case of three matrices, there is a range of interesting examples for which we can plot their Clifford spectrum using computer algebra package.  We use the the Pauli Spin matrices for the Clifford representation so that,
\begin{equation}
\gamma_1=\begin{bmatrix}
0 & 1\\
1 & 0
\end{bmatrix},
\ 
\gamma_2=\begin{bmatrix}
0 & -i\\
i & 0
\end{bmatrix},
\ 
\gamma_3=\begin{bmatrix}
1 & 0\\
0 & -1
\end{bmatrix} .\label{eq:CliffordRepThree}
\end{equation}
The localizer now becomes,
\begin{equation*}
L_{(x, y, z)}(A,B,C) 
 =  \begin{bmatrix}
  (C -zI)  & (A -xI) -i(B-yI)  \\ 
 (A -xI) + i(B-yI) & -(C -zI)
  \end{bmatrix} .
\end{equation*}

\begin{figure}
\includegraphics{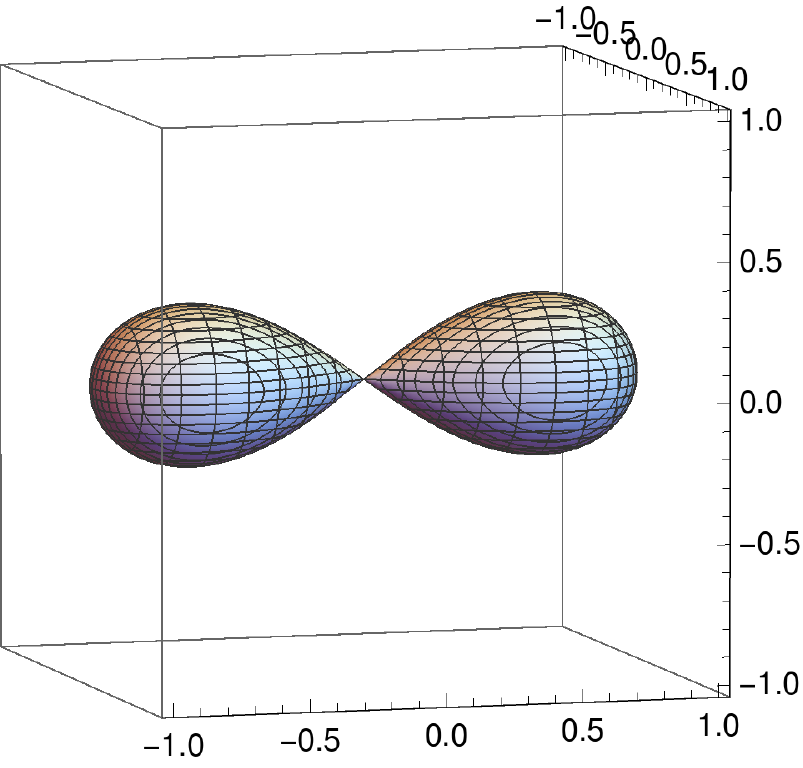} 
\caption {The Clifford spectrum of the three scaled Pauli Spin  $\Lambda \left(\tfrac{1}{2} \sigma_x, \sigma_y, \tfrac{1}{2} \sigma_z \right)$ as explained in 
Example~\ref{exa:lemniscate}.  The index at a point inside either lobe is $1$.  As always, for points on the outside the index is $0$.  See the supplemental files
\texttt{Lemniscate.*} for the calculations.
\label{fig:lem}}
\end{figure}

\begin{example}
\label{exa:Pauli_sphere}
 The first example with Clifford spectrum a surface was due to by Kisil \cite{KisilCliffordSpectrum}, and we repeat that here.  The Pauli Spin matrices themselves are the three Hermitian matrices we consider.  The following can be computed
by hand, but using symbolic algebra is preferred.  We find
\begin{equation*}
\text{char} \left(\sigma_x, \sigma_y, \sigma_z \right) 
 = (x^2 + y^2 + z^2 - 1)(x^2 + y^2 + z^2 +3)
\end{equation*}
and that here the Clifford spectrum is the unit sphere. 
\end{example}

\begin{example}
\label{exa:lemniscate}
A slight modification of the previous example leads to the Clifford spectrum
being a surface but not a manifold.  We simply rescale  some of the Pauli Spin matrices and consider $\tfrac{1}{2} \sigma_x$,  $\sigma_y$ and $\tfrac{1}{2} \sigma_z$.  The characteristic polynomial is now
\begin{equation*}
\textnormal{char} \left(\tfrac{1}{2} \sigma_x, \sigma_y, \tfrac{1}{2} \sigma_z \right)
= (x^2 + y^2 + z^2 )^2 + 2z^2 +2x^2 - y^2 
\end{equation*} 
Since $(x^2 + y^2)^2 + 2x^2 - y^2 = 0$ describes a lemniscate of Bernoulli, the surface here is a rotated lemniscate as illustrated by Figure  \ref{fig:lem}.
\end{example}

Mathematica and other computer algebra programs can produce accurate and compelling pictures of the Clifford  spectrum in many examples, but there are limitations.
Some rather simple examples can lead to the plot being incomplete, as we
will demonstrate.  We are asking a computer to verify that a certain set is infinite, which is too big of a request.  Two methods are available to verify the results of some
examples.  The first is to factor the characteristic polynomial and identify the zero-sets of the factors, which might be impossible.  The second is to employ the information we get from the $K$-theory indices associated to almost commuting matrices \cite{LoringPseudospectra}.  These generally must be zero when the Clifford spectrum is a finite set, so
calculating a single index can tell us that that a certain spectrum  is an infinite set.

The index we start with is the most basic of those introduced in \cite{LoringPseudospectra}.  It is defined in terms of the signature.  For an invertible Hermitian matrix, the signature is the the number of positive eigenvalues, minus the number of negative eigenvalues, of that matrix.

\begin{defn}
The \emph{index at}  $\boldsymbol{\lambda}$ for an ordered triple of non-commuting Hermitian matrices $X_1, X_2, X_3$ is defined only when $\boldsymbol{\lambda}$ is not in $\Lambda(X_1, X_2, X_3),$ and is given by,
\[
\textnormal{Ind}_{\boldsymbol{\lambda}} (X_1, X_2, X_3) = \dfrac{1}{2} \textnormal{Sig} \left(L_{\boldsymbol{\lambda}} (X_1, X_2, X_3) \right).
\]
\end{defn} 

\begin{figure}
\includegraphics{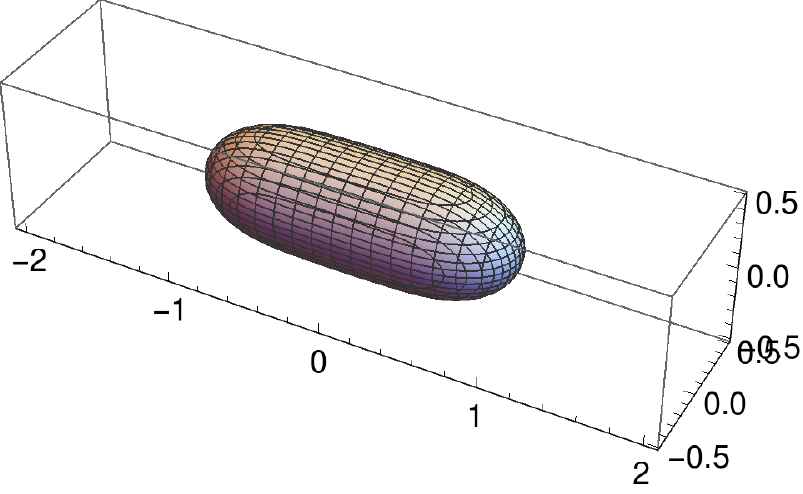}
\includegraphics{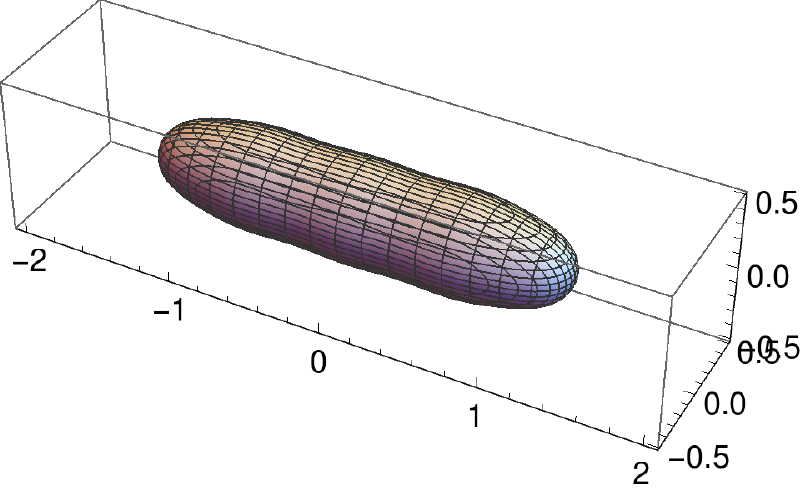}\\
\includegraphics{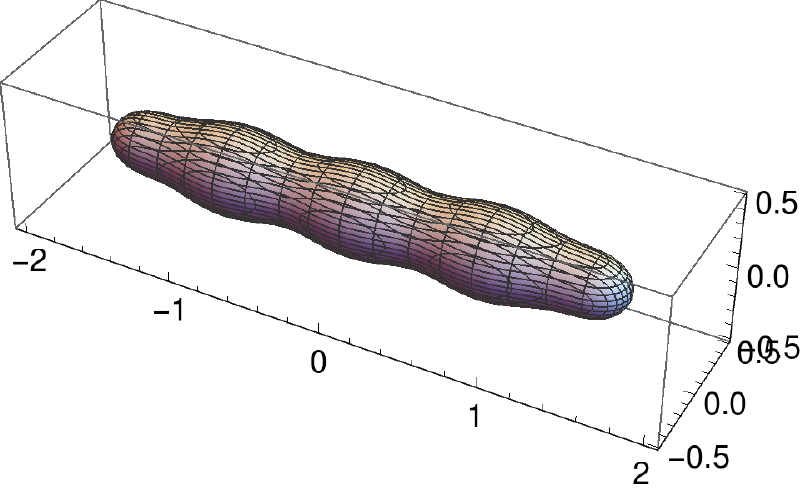}
\includegraphics{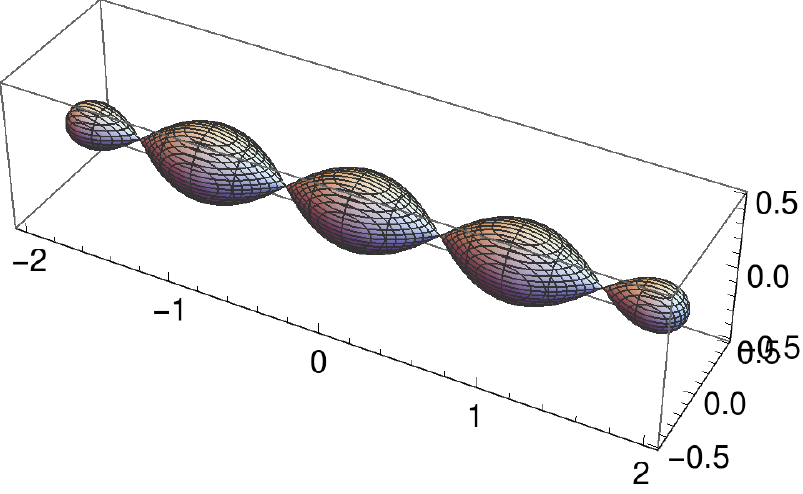}\\
\caption{The Clifford spectrum of the matrices from Example~\ref{exa:higher_lemniscate}.  The values of $t$, starting at the top-left, $\tfrac{1}{2}$, $t=\tfrac{2}{3} \dots $, $t=\tfrac{5}{6} $, $t=1$. See supplemental files \texttt{FSscaled5Croped.*} as well as the 
video file  \texttt{C5scale\_cmprsd.avi} that shows the Clifford spectrum at many more points on the path.
 \label{fig:FS5}} 
\end{figure}

The index at the origin it $1$ for the Pauli spin matrices, as in 
Example~\ref{exa:Pauli_sphere}.  
Inside either lobe of the lemniscate example this index is also $1$.
These facts can be calculated by hand, or one can
see the supplemental files \texttt{PauliSpinTwoSphere.*} and \texttt{Lemniscate.*} for the calculations.

Consider a path $\boldsymbol{\lambda}_t$ in $\mathbb{R}^3$ with fixed $X_1, X_2, X_3$, and
assume that
\begin{equation*}
\textnormal{Ind}_{\boldsymbol{\lambda}_{t_0}} (X_1, X_2, X_3)
\neq
\textnormal{Ind}_{\boldsymbol{\lambda}_{t_1}} (X_1, X_2, X_3).
\end{equation*}
Since the localizer is Hermitian, the only for this change to occur is if
the localizer becomes singular at some intermediate $t$.  Thus any path between two 
points with differing index must cross the Clifford spectrum.

It is easy to prove that if $\boldsymbol{\lambda}$ is larger than
$\|L_{\mathbf{0}}{(X_1, X_2, X_3)}\|$ then the index at $\boldsymbol{\lambda}$
 equals zero.
Thus proving that the index to be nonzero at a single point shown
that the Clifford spectrum separates that point from infinity.
This proves that in that instance the Clifford spectrum is
not a finite set.

Already with $2$-by-$2$ matrices, we start to see interesting topology emerge.  Moving up to  $5$-by-$5$ and  $6$-by-$6$ and taking paths of Hermitian matrices, we see the suggestions of  interesting patterns.  Here we present some of what we found.  We encourage the reader to use our Mathematica supplemental files, or the  SageMath code listing in \cite{sykora2016fuzzy}, as a basis to explore more examples.

\begin{figure}
\includegraphics{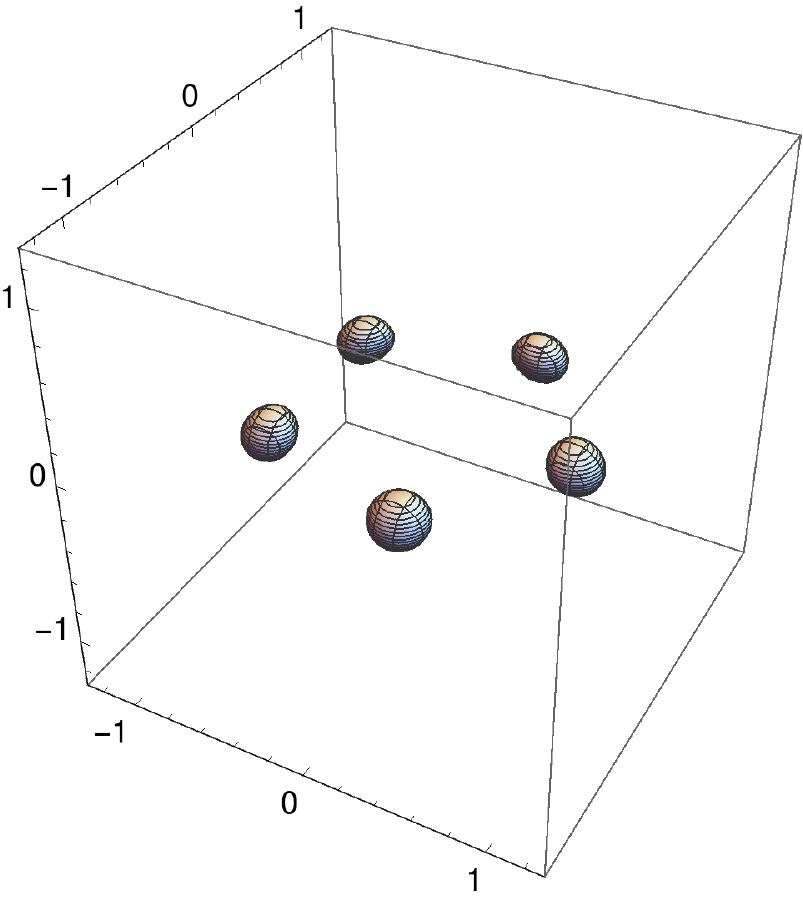}
\includegraphics{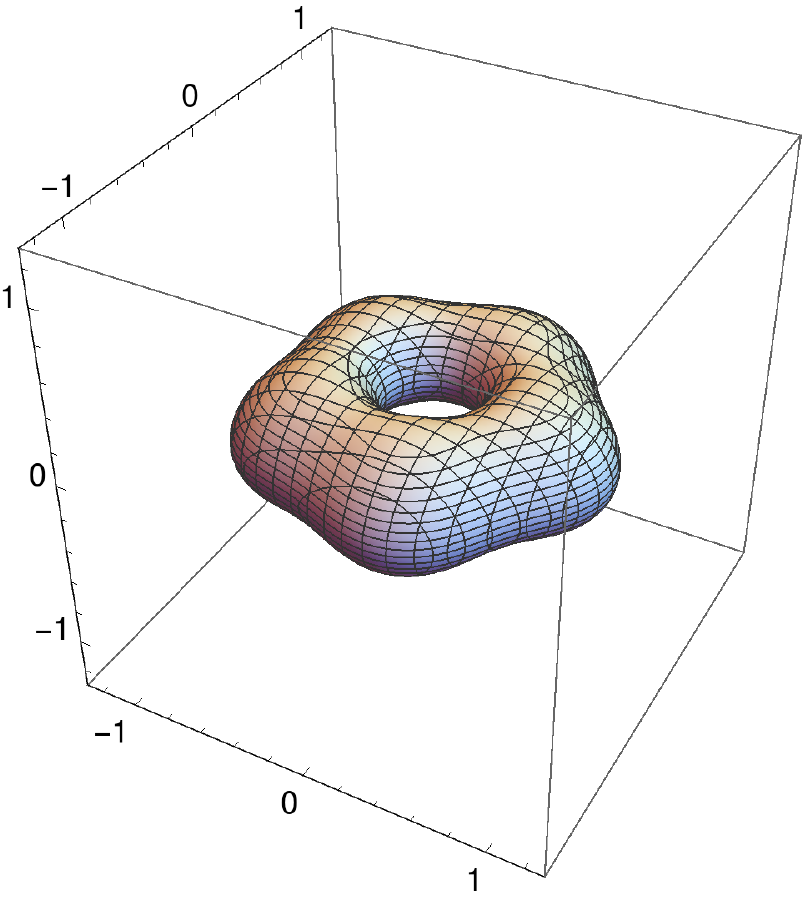}\\
\includegraphics{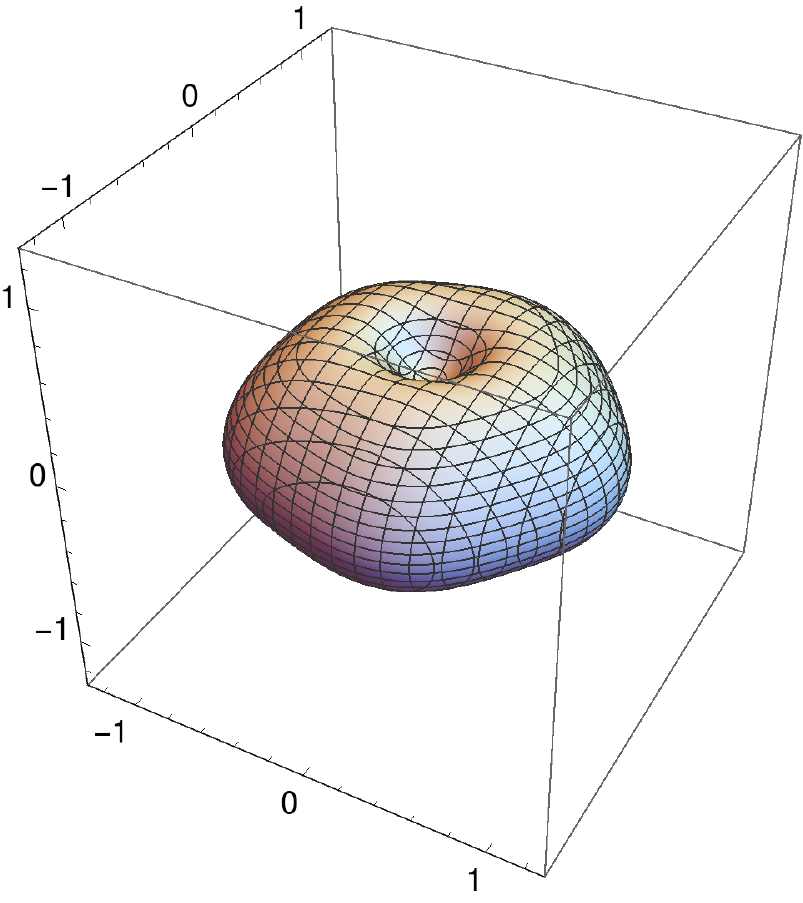}
\includegraphics{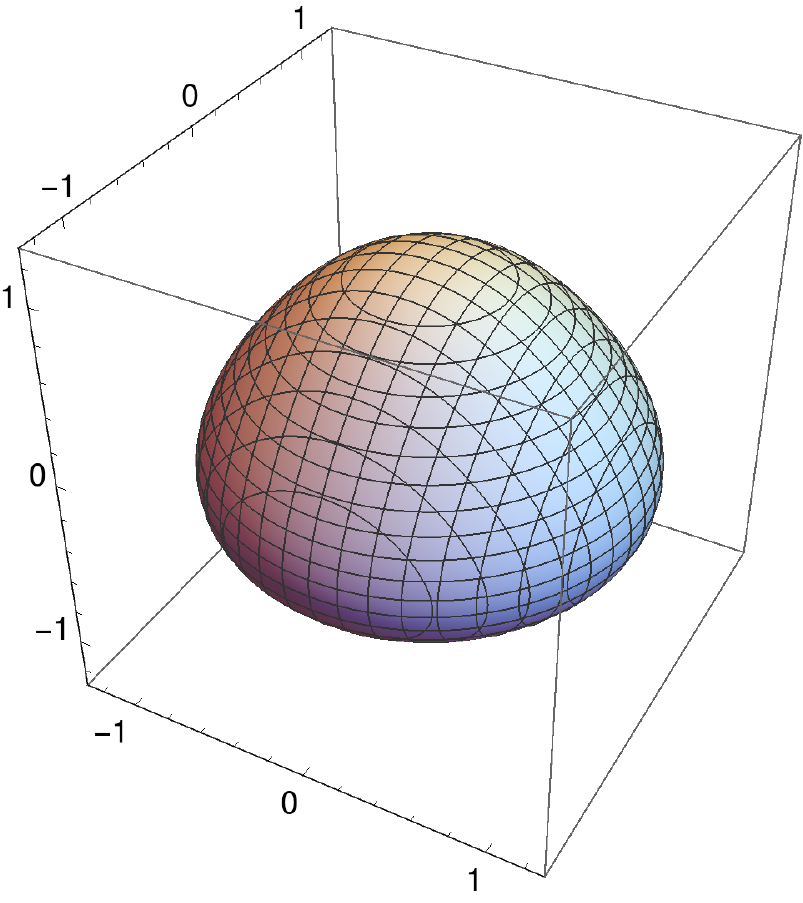}
\caption {The Clifford spectrum of of Example~\ref{exa:torus2sphere} for $R=0.9$ and from top-left, $r=\tfrac{4}{10}$, $r=\tfrac{6}{10}$, $r=\tfrac{7}{10},$ and $r=1$. The Clfford spectrum for more values of $r$ can be seen in the supplementary file \texttt{F5scale\_cmprsd.*}, and the code to create these plots is in the supplementary file \texttt{FSscaled5Croped.avi}.
\label{fig:TS5}}
\end{figure}

\begin{example}
\label{exa:higher_lemniscate}
Berenstein, Dzienkowski, and Lashof-Regas \cite{berenstein2012matrix,berenstein2015spinning} looked at
the matrices generating a fuzzy sphere.  We consider here similar matrices,
\[
A =  \begin{bmatrix} 2 & 0 & 0 & 0 & 0 \\  0 & 1 & 0 & 0 & 0 \\  0 & 0 & 0 & 0 & 0 \\  0 & 0 & 0 & -1 & 0 \\  0 & 0 & 0 & 0 & -2 \end{bmatrix},
\ 
 B =  \begin{bmatrix} 0 & \tfrac{1}{4} & 0 & 0 & 0 \\  \tfrac{1}{4}  & 0 & \tfrac{1}{4}  & 0 & 0 \\  0 & \tfrac{1}{4} & 0 & \tfrac{1}{4} & 0 \\  0 & 0 & \tfrac{1}{4}  & 0 & \tfrac{1}{4}  \\  0 & 0 & 0 & \tfrac{1}{4}  & 0 \end{bmatrix},
\ 
 C =  \begin{bmatrix} 0 & -\tfrac{i}{4} & 0 & 0 & 0 \\  \tfrac{i}{4} & 0 & -\tfrac{i}{4}  & 0 & 0 \\  0 & \tfrac{i}{4} & 0 & -\tfrac{i}{4} & 0 \\  0 & 0 & \tfrac{i}{4}  & 0 &-\tfrac{i}{4}  \\  0 & 0 & 0 & \tfrac{i}{4}  & 0 \end{bmatrix}.
\]
By rescaling one of these matrices, we were able to see a higher iteration of the lemniscate surface.  Specifically we looked along the path $(tA,B,C).$  We show in Figure~\ref{fig:FS5} the Clifford spectrum at some points along this path.
\end{example}

\begin{figure}
\includegraphics{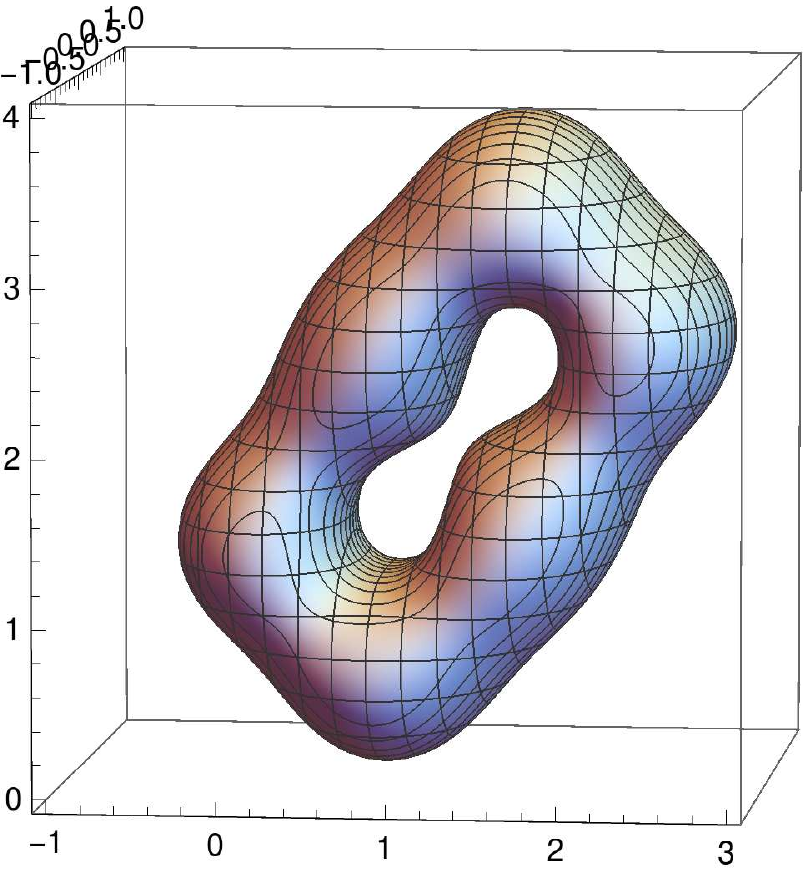}
\includegraphics{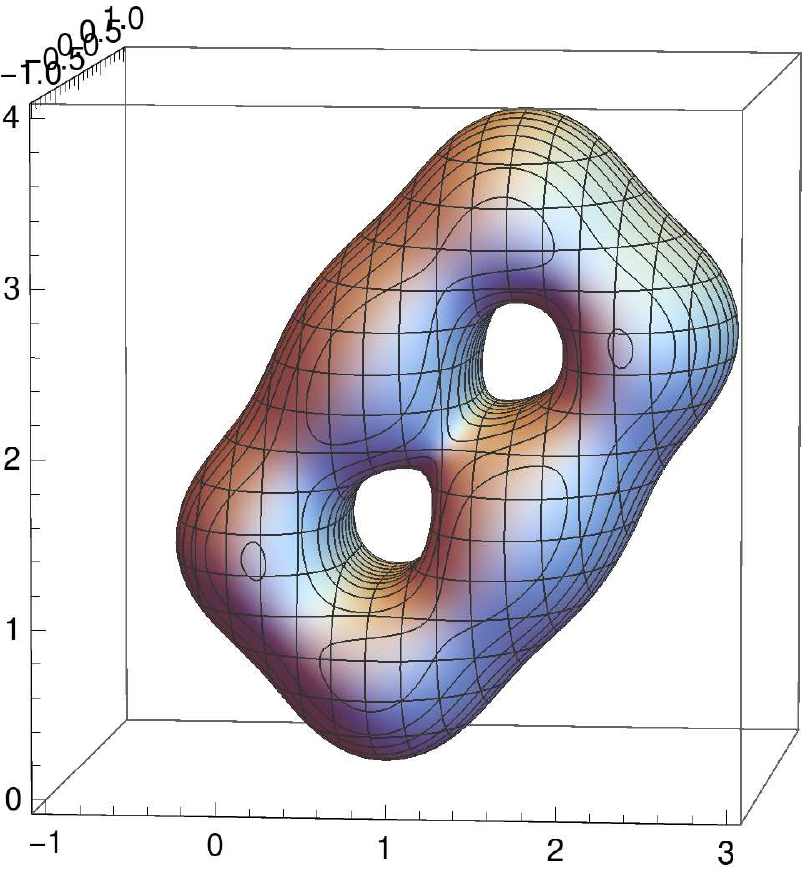}\\
\includegraphics{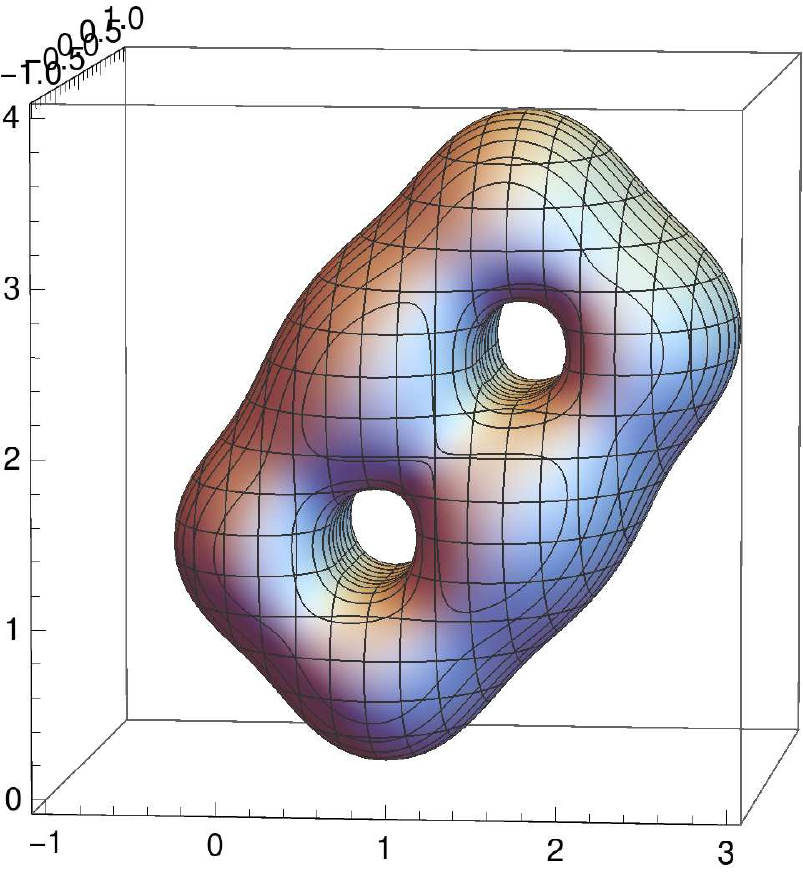}
\includegraphics{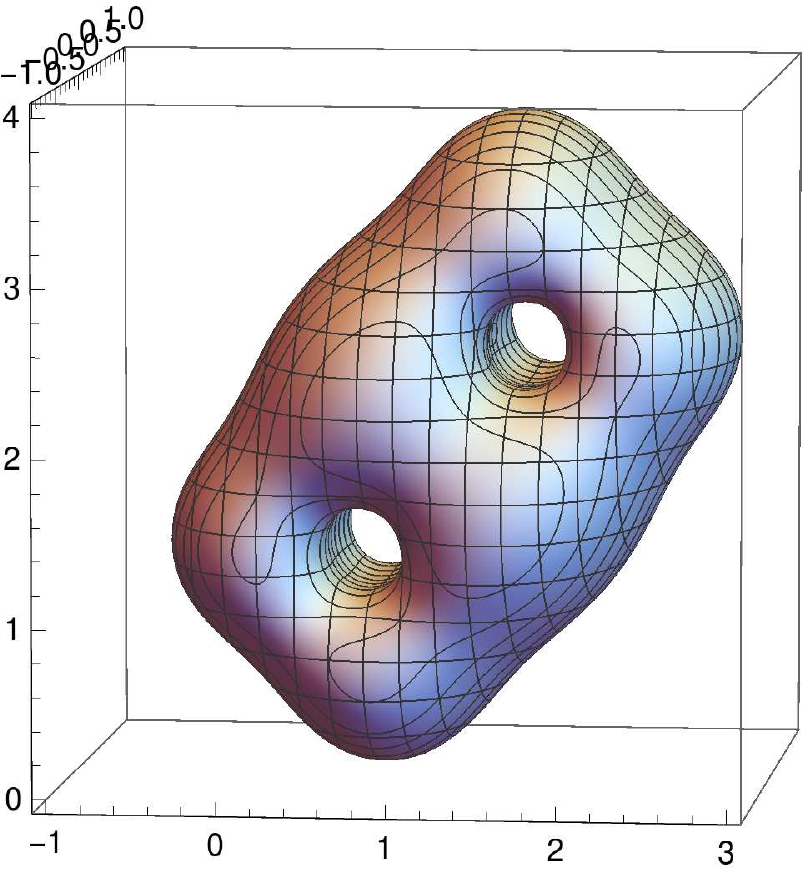}\\
\caption{A two-holed torus, and a deformation of that, arising as the Clifford
spectrum of the three matrices in Example~\ref{exa:changing_genus}. 
Starting at the top-right, the values of $r$ used are  $r=\tfrac{1}{2}$, $r=\tfrac{2}{3}$,
 $r=\tfrac{5}{6}$,  $r=1$.  The Clfford spectrum for more values of $r$ can be seen in the supplementary file \texttt{lowering\_genus\_cmprsd.avi}, and the code to create these plots is in the supplementary files \texttt{two\_holes.*}.  The index calculation in shown in supplementary files \texttt{two\_holes\_6\_index.*}. 
\label{fig:two_holes_to_one}}
\end{figure}

\begin{example}
\label{exa:torus2sphere}
This example is similar to one in \cite{berenstein2015spinning}, illustrating a transition in the Clifford spectrum between a torus and a sphere. As we want a torus, it is not surprising we start with the clock $V=V_n$ and shift $U=U_n$ unitaries from (\ref{eqn:define_U}) and (\ref{eqn:define_V}).  In Section~\ref{sec:Four-Hermitian-matrices} we will consider Clifford spectrum of four Hermitian matrices and see again a torus.  Here we want three matrices, so inspired by the usual parameterization of a torus embedded in three-space we define
\begin{eqnarray*}
\begin{aligned}
A &= \tfrac{1}{2} \left(R +  \tfrac{r}{2} U^* + \tfrac{r}{2} U \right) V^* 
+ \tfrac{1}{2} V \left(R+ \tfrac{r}{2} U^* +\tfrac{r}{2} U \right) \\
B &=\tfrac{i}{2} \left(R +  \tfrac{r}{2} U^* + \tfrac{r}{2} U \right) V^* - \tfrac{i}{2} V \left(R+ \tfrac{r}{2} U^* +\tfrac{r}{2} U\right) \\
C &=  \tfrac{ri}{2} U^* -  \tfrac{ri}{2}U.
\end{aligned}
\end{eqnarray*}
We compute this specifically with $n=5$, outer radius $R=0.9$ and variable inner radius $r$.
For four values of $r$ lead to the Clifford spectrum shown in Figure~\ref{fig:TS5}.
\end{example}

\begin{example}
\label{exa:changing_genus}
Taking a hint from \cite{sykora2016fuzzy}  we consider
\begin{eqnarray*}
\begin{aligned}
X&=\begin{bmatrix}
\frac{4}{5} & \frac{1}{2} & \frac{1}{2}\\
\frac{1}{2} & 0 &  & \frac{1}{2}\\
\frac{1}{2} &  & \frac{8}{5} & \frac{r}{2} & \frac{1}{2}\\
 & \frac{1}{2} & \frac{r}{2} & \frac{4}{5} &  & \frac{1}{2}\\
 &  & \frac{1}{2} &  & \frac{12}{5} & \frac{1}{2}\\
 &  &  & \frac{1}{2} & \frac{1}{2} & \frac{8}{5}
\end{bmatrix},
\quad 
Y=\begin{bmatrix}
0 & -\frac{i}{2} & -\frac{i}{2}\\
\frac{i}{2} & 0 &  & -\frac{i}{2}\\
\frac{i}{2} &  & 0 & -\frac{ir}{2} & -\frac{i}{2}\\
 & \frac{i}{2} & \frac{ir}{2} & 0 &  & -\frac{i}{2}\\
 &  & \frac{i}{2} &  & 0 & -\frac{i}{2}\\
 &  &  & \frac{i}{2} & \frac{i}{2} & 0
\end{bmatrix}\\
Z& =\begin{bmatrix}
0\\
 & \frac{13}{10}\\
 &  & \frac{13}{10}\\
 &  &  & \frac{13}{5}\\
 &  &  &  & \frac{13}{5}\\
 &  &  &  &  & \frac{39}{10}
\end{bmatrix}
\end{aligned}
\end{eqnarray*}
which, for $r=1$ is the smallest triples of matrices Sykora found that had
Clifford spectrum a two-holed torus.  
We computed numerically that index is  for $r=1$ at $(2, 0, 0.25)$ inside the two-holed torus to confirm we actually have a surface and not a cloud of points. 
The plots of the Clifford spectrum for several values of $r$ are shown in
Figure~\ref{fig:two_holes_to_one} .
\end{example}

\section{Four Hermitian matrices } \label{sec:Four-Hermitian-matrices}

We need to make a choice of $\gamma_1, \dots, \gamma_4$, and warn the reader that these are related to but not equal to the Dirac matrices.  The Dirac matrices square sometimes to $1$ and sometimes to $-1$.  Here we need the relations (\ref{eq:gamma_rep}) which
dictate that the matrices are all Hermitian and square to $1$.  Moreover, we have no use
for a $\gamma_0$ as we just want a linearly independent set.  We use the Pauli spin matrices for convenience, but there is no connection here with the spin of a particle.

Our choice here is as follows.
\begin{equation}
\label{eqn:four_gammas}
\begin{aligned}
\gamma_1 &= \sigma_x \otimes (-\sigma_y) 
= \begin{bmatrix}
0 & 0 & 0 & i \\ 
0 & 0 & i & 0 \\ 
0 & -i & 0 & 0 \\ 
-i & 0 & 0 & 0 \\ 
\end{bmatrix}
, \quad
\gamma_2 =  \sigma_y \otimes (-\sigma_y) 
= \begin{bmatrix}
0 & 0 & 0 & -1 \\ 
0 & 0 & 1 & 0 \\ 
0 & 1 & 0 & 0 \\ 
-1 & 0 & 0 & 0 \\ 
\end{bmatrix}
\\
\gamma_3 &=  \sigma_z \otimes (-\sigma_y) 
= \begin{bmatrix}
0 & 0 & i & 0 \\ 
0 & 0 & 0 & -i \\ 
-i & 0 & 0 & 0 \\ 
0 & i & 0 & 0 \\ 
\end{bmatrix}
, \quad
\gamma_4 =  I_2 \otimes (\sigma_x) 
= \begin{bmatrix}
0 & 0 & 1 & 0 \\ 
0 & 0 & 0 & 1 \\ 
1 & 0 & 0 & 0 \\ 
0 & 1 & 0 & 0 \\ 
\end{bmatrix}
\end{aligned}
\end{equation}
The advantage these have is each $\gamma_j$ is block off-diagonal.  We can thus define
the \emph{reduced localizer}
\begin{equation} 
\label{eqn:reduced_localizer}
\widetilde{L}_{\boldsymbol{\lambda}}(X_1,X_2,X_3,X_4) 
 =  \sum_{k=1}^3  (X_k - \lambda_k) \otimes \widetilde{\gamma}_k 
\end{equation}
in terms of the upper-right blocks of the $\gamma_j$.  Thus
\begin{equation*} 
\widetilde{\gamma}_1 = i\sigma_x, \ 
\widetilde{\gamma}_2 = i\sigma_y, \ 
\widetilde{\gamma}_3 = i\sigma_z, \ 
\widetilde{\gamma}_4 = I_2.
\end{equation*}
With this notation, the localizer becomes
\begin{equation*} 
\widetilde{L}_{\boldsymbol{\lambda}}(X_1,X_2,X_3,X_4) 
=
\begin{bmatrix}
0 & \widetilde{L}_{\boldsymbol{\lambda}}(X_1,X_2,X_3,X_4) \\
(\widetilde{L}_{\boldsymbol{\lambda}}(X_1,X_2,X_3,X_4))^* & 0
\end{bmatrix}
\end{equation*}
and the characteristic polynomial can be computed via the formula
\begin{equation*} 
\left|\textnormal{char}_{\boldsymbol{\lambda}}(X_1,X_2,X_3,X_4) \right|
= \left| \det\left( \widetilde{L}_{\boldsymbol{\lambda}}(X_1,X_2,X_3,X_4)  \right) \right|^2
\end{equation*}
Thus we have what we call the \emph{reduced characteristic polynomial} 
\begin{equation*} 
\det\left( \widetilde{L}_{\boldsymbol{\lambda}}(X_1,X_2,X_3,X_4)  \right)
\end{equation*}
and we can compute the Clifford spectrum by setting that to zero.
In computer calculations, especially, we use $(w,x,y,z)$ in place of $(\lambda_1,\lambda_2,\lambda_3,\lambda_4)$.

We have a three examples, with Clifford spectrum zero-dimensional, two-dimensional, and three-dimensional.  The case of two-dimensional Clifford spectrum in four-space is the most difficult, as such a spectrum will not separate a point from infinity.  This means there will be no possible $K$-theory argument, and we are stuck with examining a complicated characteristic polynomial.  The significance of the reduced characteristic polynomial is that cuts down by half the degree of the polynomial we must study.

To get a torus in four space, we are able to use the Hermitian and anti-Hermitian parts of
the clock and shift unitaries.  These are all symmetric matices (equal under the transpose 
$(\mbox{--})^\mathrm{T}$) except the imaginary part of the shift, which is anti-symmetric.
The following lemma helps simplify things with that symmetry.

\begin{lem}
\label{lem:poly_symmetry_from_symmetry}
Suppose that $X_1,X_2,X_3,X_4$ are Hermitian matrices, that $X_1$, $X_3$ and $X_4$ are symmetric and $X_2$ is anti-symmetric.  Then 
\begin{equation*}
\det\left( \widetilde{L}_{(\lambda_1,-\lambda_2,\lambda_3,\lambda_4)}(X_1,X_2,X_3,X_4)  \right)
=
\det\left( \widetilde{L}_{(\lambda_1,\lambda_2,\lambda_3,\lambda_4)}(X_1,X_2,X_3,X_4)  \right).
\end{equation*}
\end{lem}

\begin{proof}
We observe that 
\begin{equation*} 
\widetilde{\gamma}_k ^\mathrm{T}  = 
\begin{cases}
      \widetilde{\gamma}_k & \mbox{if } k\neq 2 \\
    - \widetilde{\gamma}_k & \mbox{if } k=2
\end{cases}
\end{equation*}
and similarly we have the assumption
\begin{equation*} 
X_k ^\mathrm{T} = 
\begin{cases}
      X_k & \mbox{if } k\neq 2 \\
    - X_k & \mbox{if } k=2.
\end{cases}
\end{equation*}
so we get that every term $X_k \otimes \widetilde{\gamma}_k $ is symmetric.  On
the other hand, every term $\lambda_k I_n \otimes \widetilde{\gamma}_k$ is
symmetric except for $k=2$, where that term is anti-symmetric.
Let $\epsilon_j = 1$ except for  $\epsilon_2 = -1$.
Then we have
\begin{align*}
\left(\sum_{k=1}^3  (X_k - \lambda_k) \otimes \widetilde{\gamma}_k \right)^\mathrm{T}
&=
\left(\sum_{k=1}^3  X_k \otimes \widetilde{\gamma}_k \right)^\mathrm{T}
  +
  \left(\sum_{k=1}^3  \lambda_k I_k \otimes \widetilde{\gamma}_k \right)^\mathrm{T}\\
&=   \sum_{k=1}^3  (X_k - \lambda_k) \otimes \widetilde{\gamma}_k 
  + \sum_{k=1}^3  \epsilon_j \lambda_k I_k \otimes  \widetilde{\gamma}_k  \\
&= \sum_{k=1}^3  (X_k - \epsilon_j \lambda_k) \otimes \widetilde{\gamma}_k
\end{align*}
Since the transpose does not effect the determinant, the result follows.
\end{proof}

\begin{table}
\begin{tabular}{|c||c|}
\hline 
$n$ & Imaginary part of reduced characteristic polynomial\tabularnewline
\hline 
\hline 
$3$ & $\left(w^{2}+x^{2}-y^{2}-z^{2}\right)\left(\frac{3}{2}\sqrt{3}\right)$\tabularnewline
\hline 
\hline 
$4$ & $\left(w^{2}+x^{2}-y^{2}-z^{2}\right)\left(4w^{2}+4x^{2}+4y^{2}+4z^{2}+8\right)$\tabularnewline
\hline 
\hline 
$5$ & $\left(w^{2}+x^{2}-y^{2}-z^{2}\right)\left(\frac{5}{2}\sqrt{\frac{1}{2}\left(65+29\sqrt{5}\right)}+[\cdots]+\frac{5}{2}\sqrt{\frac{1}{2}\left(5+\sqrt{5}\right)}z^{4}\right)$\tabularnewline
\hline 
\hline 
$6$ & $\left(w^{2}+x^{2}-y^{2}-z^{2}\right)\left(\frac{3}{2}\sqrt{3}\left(w^{2}+x^{2}+y^{2}+z^{2}+2\right)\left([\cdots]\right)\right)$\tabularnewline
\hline 
\end{tabular}
\caption{The imaginary parts of the reduced characteristic polynomials used
in the proof of Theorem~\ref{thm:torus_in_4_space}. For the full polynomials
and how they are calculated, see the supplementary files \texttt{torus\_4\_n*.*}, in particular the variable \texttt{impoly}.
\label{tab:The-imaginary-parts}}
\end{table}

\begin{table}
\begin{tabular}{|c||c|}
\hline 
$n$ & Effective real part of reduced characteristic polynomial\tabularnewline
\hline 
\hline 
$3$ & $(-2\cos(3\phi)-2\cos(3\theta))r^{3}+8r^{6}+12r^{4}+3r^{2}-1$\tabularnewline
\hline 
\hline 
$4$ & $r^{4}(-2\cos(4\phi)-2\cos(4\theta)+20)+16r^{8}+32r^{6}-4$\tabularnewline
\hline 
\hline 
$5$ & $32r^{10}+80r^{8}+\left(65+5\sqrt{5}\right)r^{6}+(-2\cos(5\phi)-2\cos(5\theta))r^{5}+[\cdots]$\tabularnewline
\hline 
\hline 
$6$ & $64r^{12}+192r^{10}+240r^{8}+(-2\cos(6\phi)-2\cos(6\theta)+148)r^{6}+9r^{4}-54r^{2}-27$\tabularnewline
\hline 
\end{tabular}\caption{Real parts  of the reduced characteristic polynomials used
in the proof of Theorem~\ref{thm:torus_in_4_space}. See the supplementary files \texttt{torus\_4\_n*.*}, in particular the variable \texttt{altpoly}.
 \label{tab:Real-parts}}
\end{table}

\begin{table}
\begin{tabular}{|c||c|}
\hline 
$n$ & Derivatives in $r$ of the Effective real parts\tabularnewline
\hline 
\hline 
$3$ & $(-6\cos(3\phi)-6\cos(3\theta))r^{2}+48r^{5}+48r^{3}+6r$\tabularnewline
\hline 
\hline 
$4$ & $r^{3}(-8\cos(4\phi)-8\cos(4\theta)+80)+128r^{7}+192r^{5}$\tabularnewline
\hline 
\hline 
$5$ & $320r^{9}+640r^{7}+\left(390+30\sqrt{5}\right)r^{5}+(-10\cos(5\phi)-10\cos(5\theta))r^{4}+[\cdots]$\tabularnewline
\hline 
\hline 
$6$ & $768r^{11}+1920r^{9}+1920r^{7}+(-12\cos(6\phi)-12\cos(6\theta)+888)r^{5}+36r^{3}-108r$\tabularnewline
\hline 
\end{tabular}\caption{Derivatives in $r$ of the function in Table~\ref{tab:Real-parts}.
\label{tab:Real-parts-Deriv}}
\end{table}

\begin{thm}
\label{thm:torus_in_4_space}
Suppose $n$ equals $3$,  $4$,  $5$ or  $6$, and define
\begin{equation*} 
\begin{aligned}
X_1 &= \tfrac{1}{2}U_n^* + \tfrac{1}{2}U_n 
, \quad
X_2 &= \tfrac{i}{2}U_n^* - \tfrac{i}{2}U_n\\
X_3 &= \tfrac{1}{2}V_n^* + \tfrac{1}{2}V_n 
, \quad
X_4 &=\tfrac{i}{2}V_n^* - \tfrac{i}{2}V_n\\
\end{aligned}
\end{equation*} 
where $U_n$ and $V_n$ are the clock and shift unitaries as in {\rm(\ref{eqn:define_U})}  and {\rm (\ref{eqn:define_V})}.  Then the Clifford spectrum of $(X_1,X_2,X_3,X_4)$ is
homeomorphic to a two-torus.
\end{thm}

\begin{proof}
We would like to solve for where the reduced localizer is zero,
\begin{equation}
\det\left( \widetilde{L}_{\boldsymbol{\lambda}}(X_1,X_2,X_3,X_4)  \right) =0.
\end{equation} 

We will do that in the following way. First, we will find the condition for the imaginary part of the localizer to be zero. Then, after setting its imaginary part to zero, we will show that the real part has both positive and negative values, which implies that it crosses zero at some point. Therefore, at the latter point both real and imaginary parts are zero, which means the whole thing is zero. 

We let used computer algebra to calculate and simplify the reduced characteristic
polynomial, with results as shown in Table~\ref{tab:The-imaginary-parts}.
In all cases, the condition $\Im \det \widetilde{L}_{(w,x,y,z)}= 0$ becomes 
\begin{equation}
 w^2+x^2=y^2+z^2.
 \label{eqn:one_radius}
\end{equation} 
We now apply Lemma~\ref{lem:poly_symmetry_from_symmetry} and deduce we have $(w,x,y,z)$ in the Clifford spectrum if,  and only if, $(w,-x,y,z)$ is the Clifford spectrum.  Thus we are justified in assuming $x \geq 0$.  With this assumption, the condition $\Im \det \widetilde{L}_{(w,x,y,z)}= 0$ becomes
\begin{equation*}
 x = \sqrt{-w^2 + y^2+z^2}.
\end{equation*} 
This means we can eliminate $x$ in the polynomial $\Re \det \widetilde{L}_{(w,x,y,z)}$ via the substitution 
\begin{equation*}x \mapsto \sqrt{-w^2 + y^2+z^2} .
\end{equation*} 
With this substitution, we get a somewhat more reasonable polynomial.  In the
case of $n=3$ it is
\begin{equation*}
\begin{gathered}
-8 w^3+3 z^2 \left(2 w+8 y \left(y^3+y\right)+2 y+1\right) \\
+6 w y^2+8 y^6+12 y^4-2 y^3 \\
+12 \left(2 y^2+1\right) z^4+3 y^2+8z^6-1
\end{gathered}
 \end{equation*}   
and for $n=4,5,6$ this polynomial has too many terms to easily display.  It
but can be seen as  \texttt{realpoly} in the supplementary files
\texttt{torus\_4\_n*.*}.

Inspired by (\ref{eqn:one_radius}) we switch to polar coordinates in the first two and also the last two variables, as we know the radius will be the same.  That is, we
make the substitution
\begin{equation} 
\begin{aligned}
w &= r \cos \theta, 
\quad
x &= r \sin \theta \\
y &= r \cos \phi ,
\quad
z &= r \sin \phi \\
\end{aligned}
\label{eqn:polar_coord}
\end{equation} 
and find the computer does a much better job simplifying.
The Clifford spectrum will be the zero set of the functions
shown in Table~\ref{tab:Real-parts},
interpreted via (\ref{eqn:polar_coord}).  
The function in the $n=5$ case was too long for the table, but can be
seen as  \texttt{altpoly} in the supplementary files
\texttt{torus\_4\_n5.*}.

Now we finish the proof for the case $n=4$, which is the easiest case.
Let's denote the relevant function from
Table~\ref{tab:Real-parts} by $f(r, \theta, \phi)$, so 
\begin{equation*} 
f(r, \theta, \phi) = -4+32r^6+16r^8+(20-2 \cos (4 \phi) - 2 \cos (4 \theta))r^4  \end{equation*}
and its $r$ derivative is 
\begin{equation*} 
\frac{\partial f}{\partial r}
= 192 r^5 +128 r^7 + (80 - 8 \cos (4 \phi) - 8 \cos (4 \theta)) r^3 .
\end{equation*} 
Since sine and cosine are bounded by $\pm 1$ we see that, for any
angles $\phi$ and $\theta$,
$\frac{\partial f}{\partial r} > 0 $
for all $r>0$ and so $f(r, \theta, \phi)$ is increasing for $r\geq 0$.
By observing that 
\begin{equation*} 
f (\theta, \phi, 0) = -4 
\end{equation*} 
and
\begin{equation*} 
\lim_{r \rightarrow \infty} f (\theta, \phi, r) = \infty 
\end{equation*}
we know that, for any fixed $(\theta, \phi)$, there exist at least one value of $r$ for which $f (\theta, \phi,r) =0$, and the fact that $\partial f/ \partial r>0$ implies that this value of $r$ is unique.
Call this value $\rho(\theta, \phi)$, so
\begin{equation*} 
f(\theta, \phi,  \rho(\theta, \phi)) = 0 
\end{equation*} 
Thus, the surface we are looking for is precisely the surface $r = \rho(\theta, \phi)$, which is
indeed topologically equivalent to a torus since $\rho(\theta, \phi)$ must
vary continuously in $\theta$ and $\phi$ since the roots of a polynomial
vary continuous with respect to the coefficients \cite{harrisRootAreContinuous}.
The resulting surface in illustrated in Figure~\ref{fig:torus_R4_n3}.

Now we look at the case $n=3$.
The relevant function from
Table~\ref{tab:Real-parts} is 
\begin{equation*} 
f(r, \theta, \phi) =  (-2\cos(3\phi)-2\cos(3\theta))r^{3}+8r^{6}+12r^{4}+3r^{2}-1
\end{equation*}
with derivative in $r$ being
\begin{equation*} 
\frac{\partial f}{\partial r} = (-6\cos(3\phi)-6\cos(3\theta))r^{2}+48r^{5}+48r^{3}+6r
\end{equation*}
For $0 < r \leq \tfrac{1}{2} $ we have the estimate
\begin{align*}
\frac{\partial f}{\partial r} & > (-6\cos(3\phi)-6\cos(3\theta))r^{2} + 6r \\
  & \geq  (-12r + 6)r  \geq 0
\end{align*}
and for $\tfrac{1}{2} \leq r \leq 1 $ we have the estimate
\begin{align*}
\frac{\partial f}{\partial r} & >  (-6\cos(3\phi)-6\cos(3\theta))r^{2}+48r^{3} \\
 & \geq (-12 +48r)r^{2}   \geq 0
\end{align*}
so again the derivative is positive except at zero it is zero.  The rest of the
proof follows as in the case $n=4$.
The resulting surface in illustrated in Figure~\ref{fig:torus_R4_n4}.

For the case $n=5$ one can prove that for $0\leq r\leq \tfrac{3}{5}$,
\begin{equation*} 
f(r, \theta, \phi) \leq -2
\end{equation*}
and,  for $\tfrac{3}{5} \leq r \leq 1 $,
\begin{equation*} 
\frac{\partial f}{\partial r}  \geq  33
\end{equation*}
so again we see that for each pair of angles there is only one radius
to make this function zero.
The work to create these two estimates is shown in the 
supplementary files
\texttt{torus\_4\_n5.*}.

For the case $n=6$ one can prove that for $0\leq r\leq \tfrac{3}{5}$,
\begin{equation*} 
f(r, \theta, \phi) \leq   -20
\end{equation*}
and,  for $\tfrac{3}{5} \leq r \leq 1 $,
\begin{equation*} 
\frac{\partial f}{\partial r}  \geq  42
\end{equation*}
so again we see that for each pair of angles there is only one radius
to make this function zero.
The work to create these two estimates is shown in the 
supplementary files
\texttt{torus\_4\_n6.*}.

\end{proof}

\begin{figure}
\includegraphics{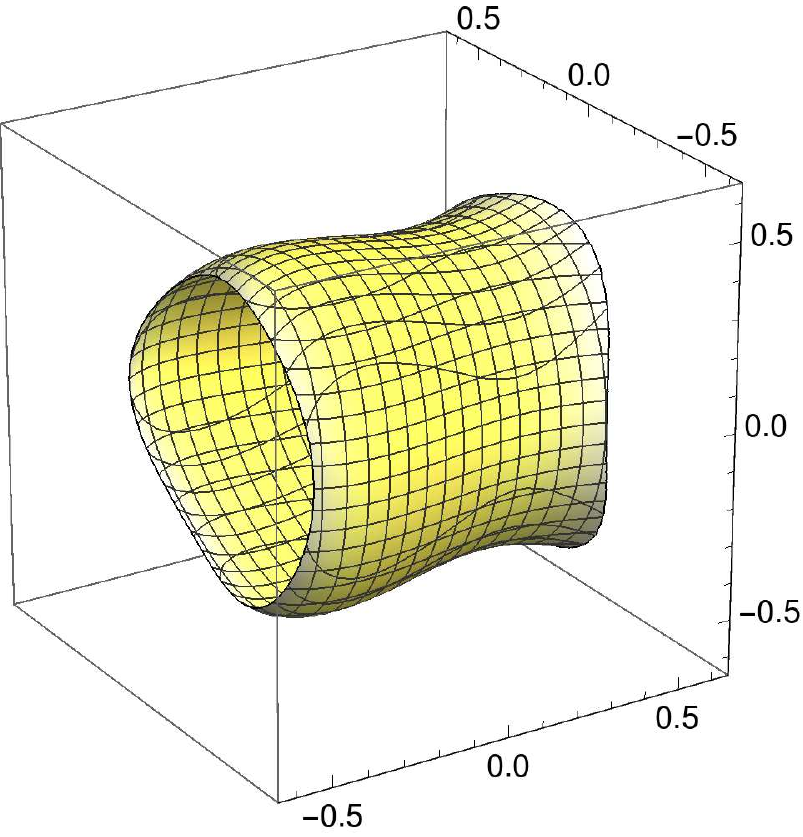}
\includegraphics{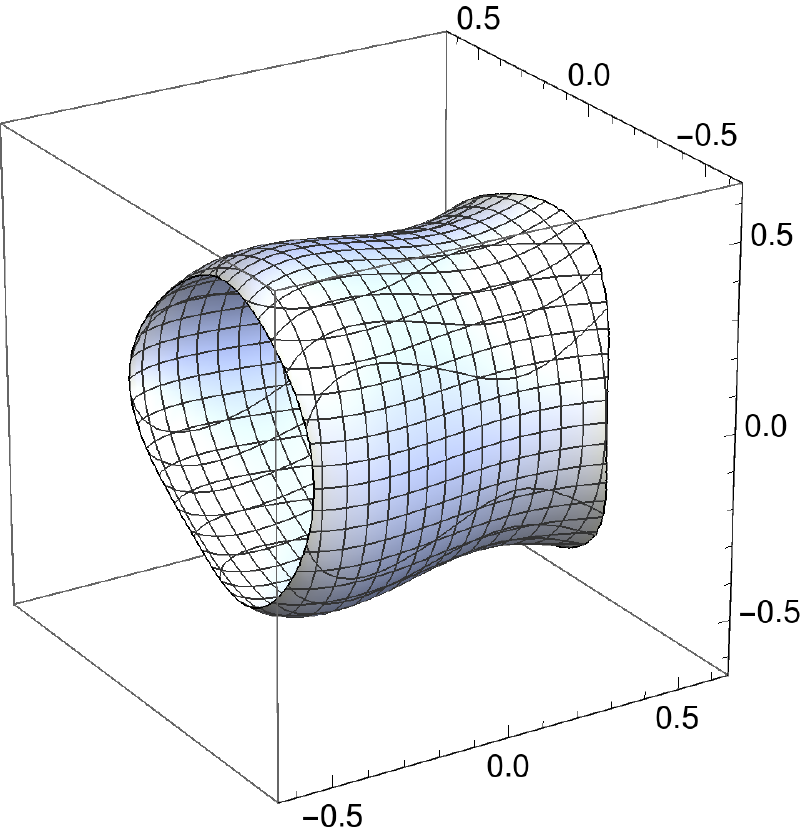}
\caption{The Clifford spectrum as a surface in four space. The top and bottom
represent half the surface, with color indicating the value in the
fourth dimension \textemdash{} white indicates zero, shades of yellow
indicate positive values, and shades of blue negative values. This
is for the for Hermitian matrices extracted from the clock and shift
matrices, with $n=3$. 
\label{fig:torus_R4_n3}}
\end{figure}

\begin{figure}
\includegraphics{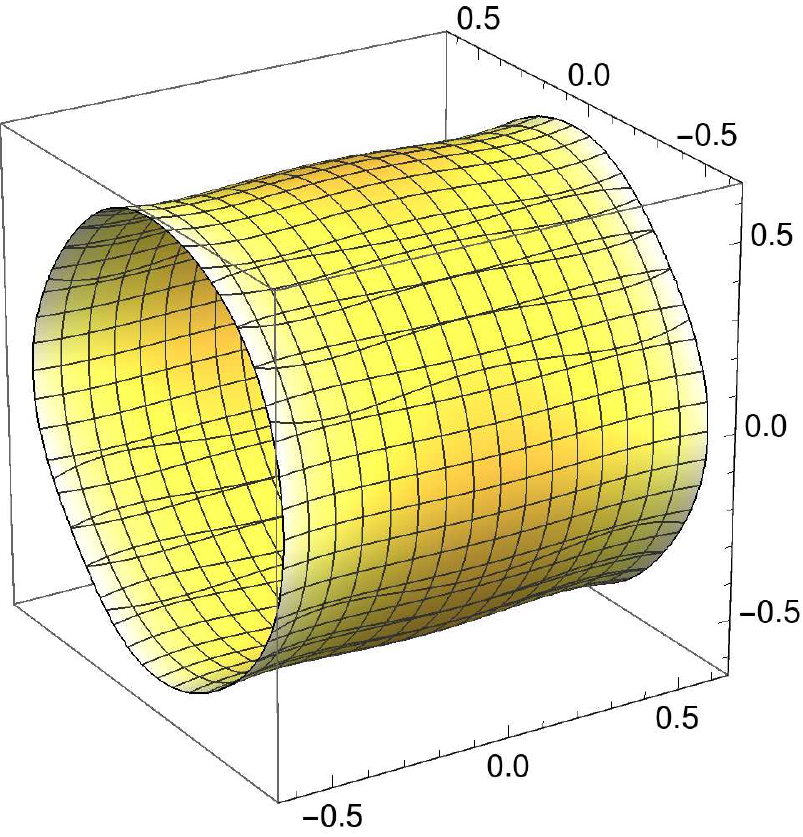}
\includegraphics{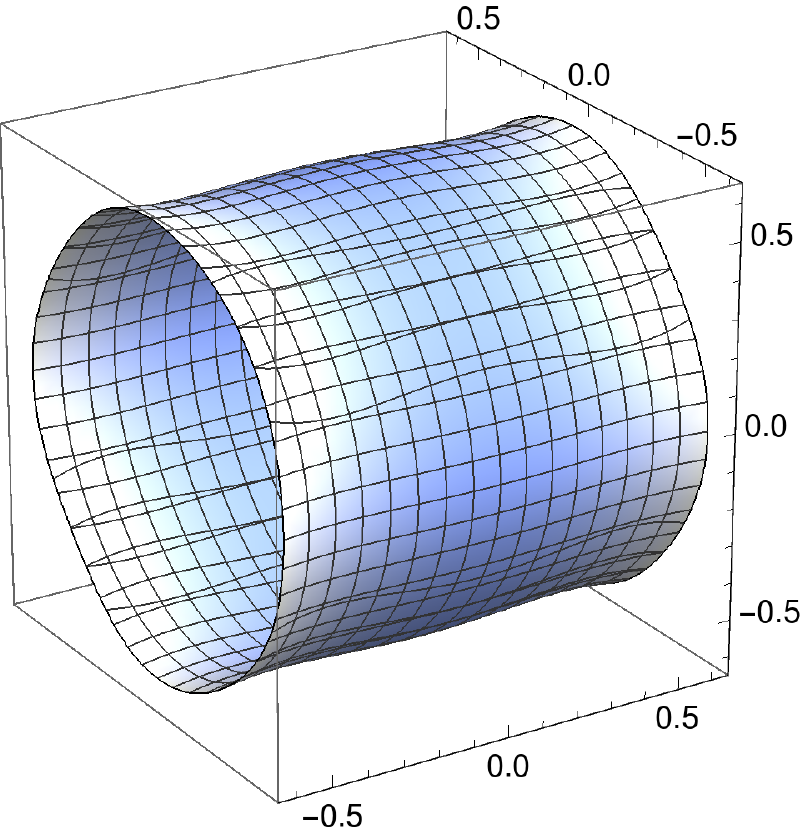}
\caption{The Clifford spectrum as a surface in four space, for the Hermitian matrices extracted from the clock and shift matrices,
with $n=4$. 
\label{fig:torus_R4_n4}}
\end{figure}

\begin{example}
In example~\ref{exa:Pauli_sphere} we saw that the Clifford spectrum of the gamma matrices lead to a sphere.  Taking the Clifford spectrum of the four gamma matrices (\ref{eqn:four_gammas}) gives a somewhat different answer.
In the supplementary file \texttt{GammaMatrices\_4B.*} is are the symbolic calculations
that for these four matrices the reduced characteristic polynomial is
\begin{equation*} 
\left(w^2+x^2+y^2+z^2\right)^3 \left(w^2+x^2+y^2+z^2+8\right)
\end{equation*}
and so the Clifford spectrum is a single point.
\end{example}

\begin{figure}
\includegraphics{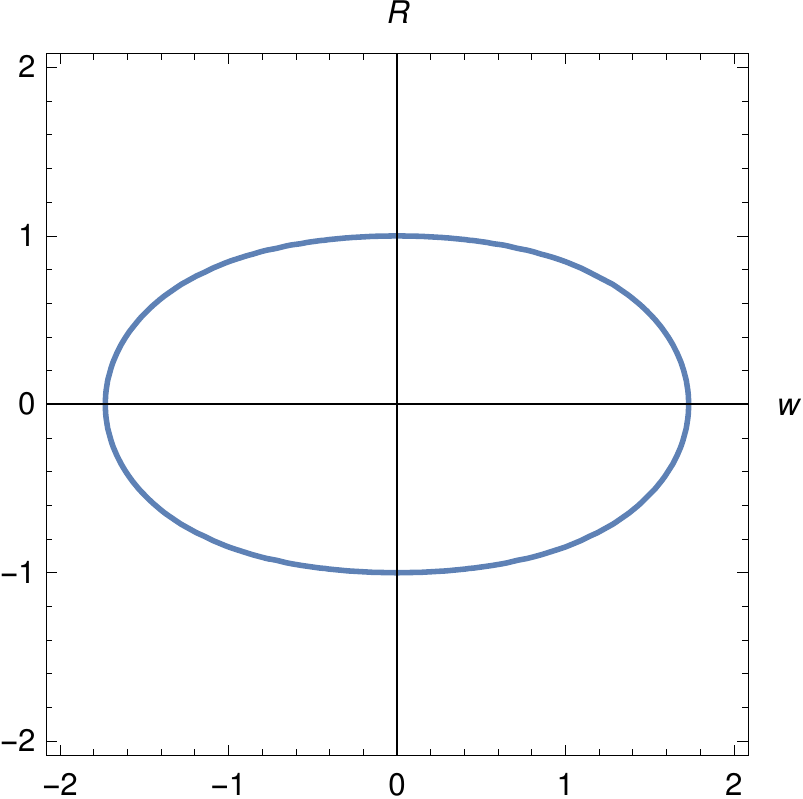}
\caption{The Clifford spectrum in Example~\ref{exa:rescaledGamma} is
this curve rotated in the two additional dimensions.
\label{fig:almostThreeSphere}}
\end{figure}

\begin{example}
\label{exa:rescaledGamma}
Now we look at a rescaling of the four gamma matrices (\ref{eqn:four_gammas}),
\begin{equation*} 
X_1 = 2 \gamma_1 ,\quad
X_2 = \gamma_2 ,\quad
X_3 = \gamma_3 ,\quad
X_4 = \gamma_4 .
\end{equation*}
and find, in supplementary file \texttt{GammaMatrices\_4A.*}, that the 
reduced characteristic polynomial is
\begin{equation*} 
(9 + 6 R^2 + R^4 - 6 w^2 + 2 R^2 w^2 + w^4) (-15 + 14 R^2 + R^4 + 
   2 w^2 + 2 R^2 w^2 + w^4)
\end{equation*}
where $R = \sqrt{x^2 + y^2 + z^2}$.  For this example, the Clifford spectrum
is homeomorphic to the three-sphere.  See Figure~\ref{exa:rescaledGamma}.
\end{example}

\section{Symmetry classes and $K$-theory charges}
\label{SymmetryClasses}

\subsection{Where the index and plotting fail}

\begin{figure}
\includegraphics{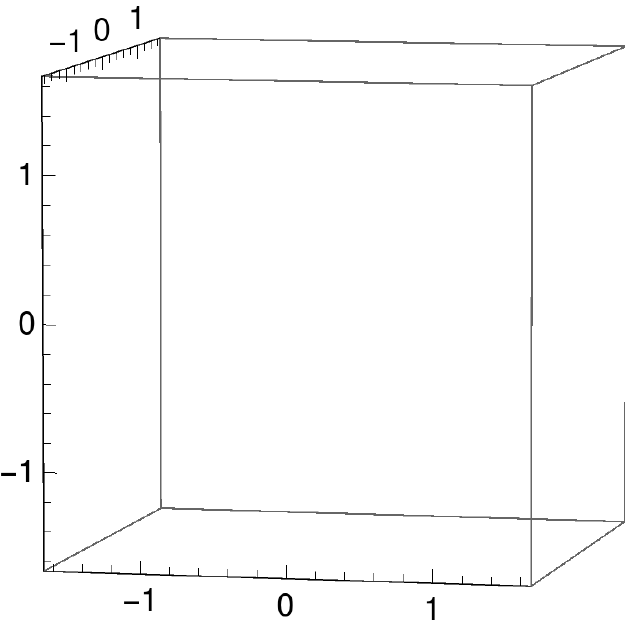}
\includegraphics{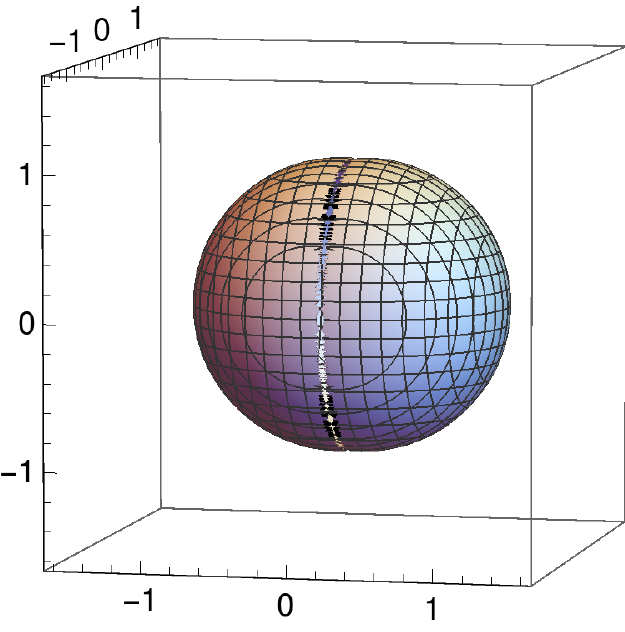}\\
\includegraphics{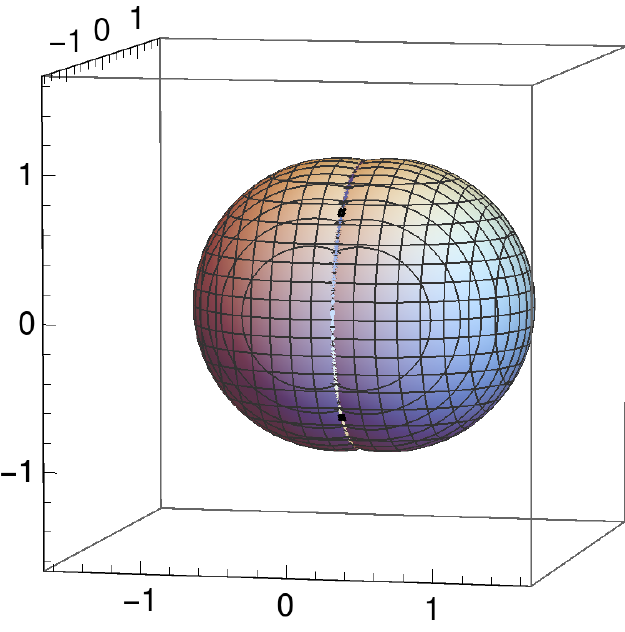}
\includegraphics{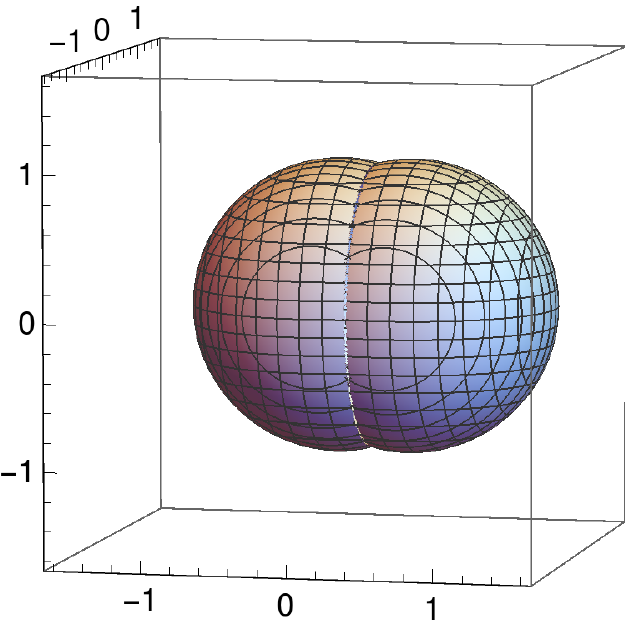}\\
\caption{An example where we cannot trust the plot via the characteristic
polynomial. This is using matrices as in Example~\ref{exa:bad_plot_example},
with $r=0$ at the top left, increasing by $1/6$ and ending at the
bottom right with $r=1/2$.
The code to create these graphics are in the
in the supplementary files \texttt{ClassAIIsphere.*}. 
\label{fig:Two_spheres_bad_plotting}} 
\end{figure}

We have the index to give us critical information about the surfaces
we have plotted. Sometimes the Clifford spectrum is a surface
but the index is zero everywhere it is defined.  Moreover, in those situations
the computer plotting can fail.

\begin{example}
\label{exa:bad_plot_example}
The three matrices we consider are as follows:
\begin{equation}
\label{eq:null_plot_example}
X=\begin{bmatrix}
0 & 1 & 0 & 0\\
1 & 0 & 0 & 0\\
0 & 0 & 0 & 1\\
0 & 0 & 1 & 0
\end{bmatrix},\,Y=\begin{bmatrix}
0 & -i & 0 & 0\\
i & 0 & 0 & 0\\
0 & 0 & 0 & i\\
0 & 0 & -i & 0
\end{bmatrix},\,Z=\begin{bmatrix}
1 & 0 & 0 & 0\\
0 & -1 & 0 & 0\\
0 & 0 & 1 & 0\\
0 & 0 & 0 & -1
\end{bmatrix} .
\end{equation}
Since the characteristic polynomial respects direct sums, it is easy to see
from Example \ref{exa:Pauli_sphere} that the characteristic polynomial
is
\begin{equation*}
\text{char} \left(\sigma_x, \sigma_y, \sigma_z \right) 
 = (x^2 + y^2 + z^2 - 1)^2(x^2 + y^2 + z^2 +3)^2
\end{equation*}
so the Clifford spectrum is the unit sphere.  Also, by looking at the
direct sum structure, one can check that the index zero at the origin.
Thus the index is zero everywhere it is defined.
Figure~\ref{fig:Two_spheres_bad_plotting}
looks at the plot Mathematica makes using the characteristic polynomial
for 
\begin{equation}
X_{r}=\begin{bmatrix}
0 & 1 & 0 & 0\\
1 & 0 & 0 & 0\\
0 & 0 & r & 1\\
0 & 0 & 1 & r
\end{bmatrix},\,Y_{r}=\begin{bmatrix}
0 & -i & 0 & 0\\
i & 0 & 0 & 0\\
0 & 0 & 0 & i\\
0 & 0 & -i & 0
\end{bmatrix},\,Z_{r}=\begin{bmatrix}
1 & 0 & 0 & 0\\
0 & -1 & 0 & 0\\
0 & 0 & 1 & 0\\
0 & 0 & 0 & -1
\end{bmatrix}\label{eq:bad_plot_matrices}
\end{equation}
for various small values of $r$, and also at zero. At zero the output
is the null plot, which is wrong.
\end{example}

\subsection{A refined index in the case of self-dual symmetry}

\begin{figure}
\includegraphics{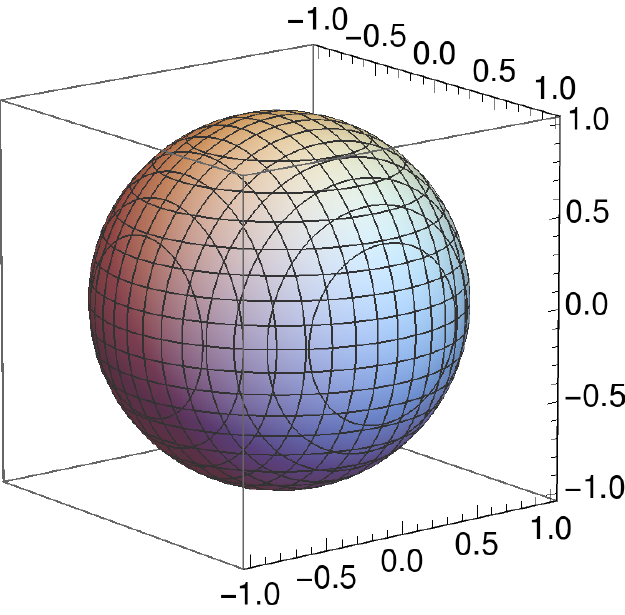}
\includegraphics{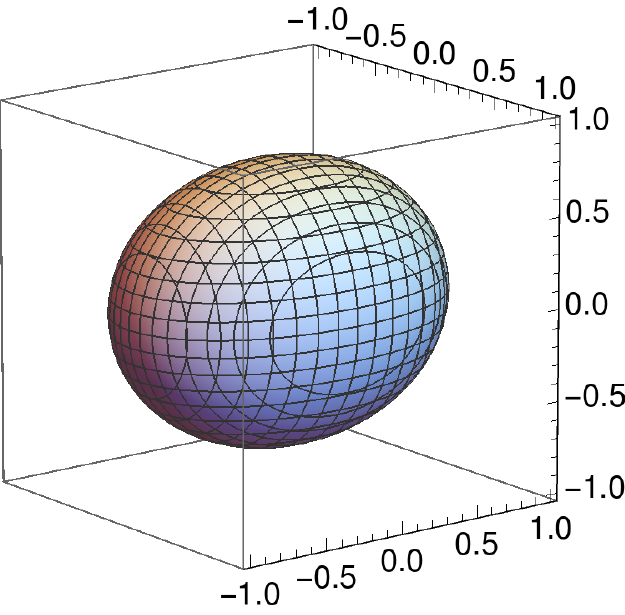}\\
\includegraphics{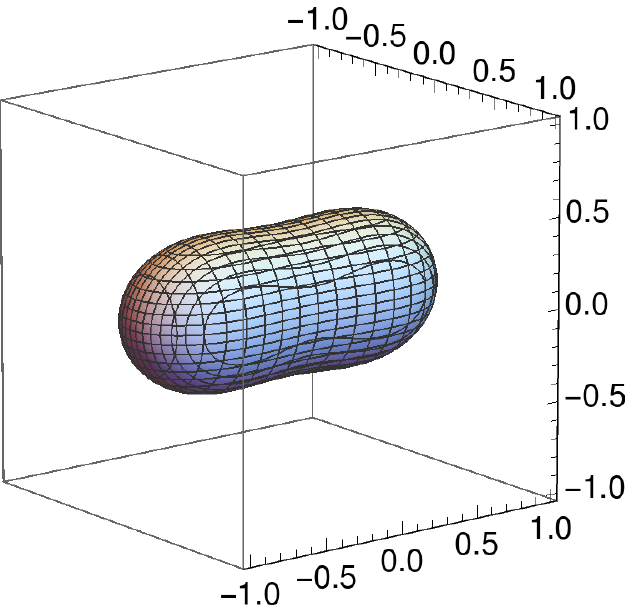}
\includegraphics{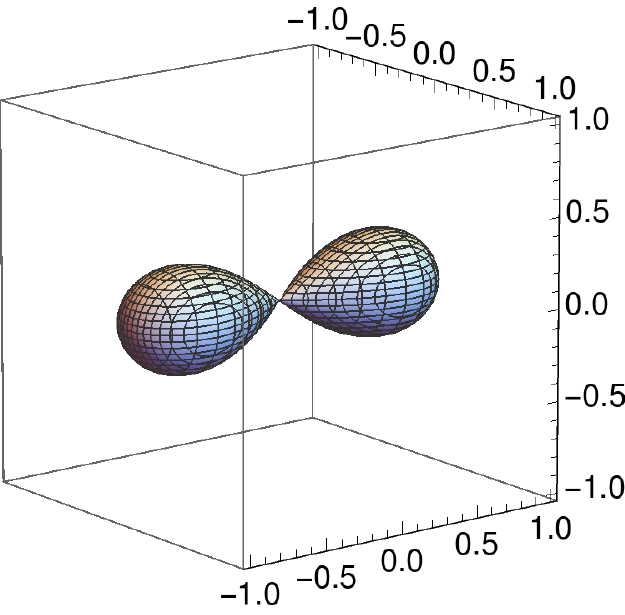}\\
\caption{The self dual matrices from Example~\ref{exa:path_of_self_dual}, plotted using the archetypal polynomial. This is using
matrices of Equation~\ref{eq:SelfDualPath}, with $s=0$ at the
top left, increasing by $1/6$ and ending at the bottom right
with $s=1/2$. Plots made using the
supplemntary file \texttt{ClassAIIspherePfaff.nb}.
\label{fig:Sphere_to_lemniscate_Pfaffian}}.
\end{figure}

In the case of the matrices in Equation~\ref{eq:null_plot_example},
the matrices had an extra symmetry that went unused. They are all self-dual, a mathematical
interpretation of having fermionic time reversal symmetry. 

Recall that the dual operation is defined as, 
\begin{equation*}
X^{\#} = {\begin{bmatrix} A & B \\ C & D \end{bmatrix}}^{\#} = \begin{bmatrix} D^T & -B^T \\ -C^T & A^T \end{bmatrix},
\end{equation*}
where $A, B, C,$ and $D$ are square complex matrices.  When a matrix $X$ is self-dual and Hermitian, we have both, $X^{\#}=X$ and $X^*=X.$

If we have three matrices that are Hermitian and self-dual, we find that the
localizer has an extra symmetry.  In this case, there is a matrix $Q$ that conjugates the spectral localizer nicely, given by
\begin{equation*}
Q = \begin{bmatrix} I_{2n} & -iZ_{2n} \\ iZ_{2n} & I_{2n} \end{bmatrix}
\end{equation*}
where
\begin{equation*}
Z_{2n} =  \begin{bmatrix} 0 & I_{n} \\ -I_{n} & 0 \end{bmatrix}.
\end{equation*}
Conjugating the spectral localizer, by the unitary matrix $\tfrac{1}{\sqrt{2}}Q$ we keep the determinant unchanged.  That is,
\begin{equation*}
\Big( \tfrac{1}{\sqrt{2}}Q \Big)^* L_{\boldsymbol{\lambda}} ( A, B, C) \Big( \tfrac{1}{\sqrt{2}}Q \Big) = \tfrac{1}{2}Q^*L_{\boldsymbol{\lambda}} ( A, B, C)Q
\end{equation*}
and 
\begin{equation*}
\text{det} \left( \tfrac{1}{2}Q^*L_{\boldsymbol{\lambda}} ( A, B, C)Q \right) = \text{det}( L_{\boldsymbol{\lambda}}(A,B,C))=\text{char}_{\boldsymbol{\lambda}}(A,B,C).
\end{equation*}
Using Lemma $8.1$ of Factorization of Matrices of Quaternions \cite{loring2012factorization} we confirm that the conjugation produces a skew-symmetric representation of the localizer and therefore,
\begin{equation*}
\left( \tfrac{1}{2}Q^*L_{\boldsymbol{\lambda}} ( A, B, C)Q \right)^\mathrm{T} = -\tfrac{1}{2}Q^*L_{\boldsymbol{\lambda}} ( A, B, C)Q 
\end{equation*}
We can now use the pfaffian instead of the determinant to 
detect where the localizer is singular. 
\begin{defn}
The \textit{archetypal polynomial} of a self-dual Hermitian triple $( X, Y, Z)$ is defined as
 \begin{equation*}
\textnormal{arch}_{\boldsymbol{\lambda}}(X, Y, Z) = \textnormal{Pf}\left( \tfrac{1}{2}Q^*L_{\boldsymbol{\lambda}} (X, Y, Z)Q \right).
\end{equation*}
\end{defn}

\begin{example}
\label{exa:path_of_self_dual}
We look at a different path that starts with 
the troublesome matrices of (\ref{eq:null_plot_example}).  For  $0\leq s\leq\tfrac{1}{2}$ we define matrices
\begin{equation}
\begin{aligned}
X_{s}= & \begin{bmatrix}
0 & 1-2s & 0 & s\\
1-2s & 0 & -s & 0\\
0 & s & 0 & 1-2s\\
-s & 0 & 1-2s & 0
\end{bmatrix}
,\quad
Y_{s}=  \begin{bmatrix}
0 & -i & 0 & 0\\
i & 0 & 0 & 0\\
0 & 0 & 0 & i\\
0 & 0 & -i & 0
\end{bmatrix},                  \\
Z_{s}= &\begin{bmatrix}
1-s & 0 & 0 & 0\\
0 & -1+s & 0 & 0\\
0 & 0 & 1-s & 0\\
0 & 0 & 0 & -1+s
\end{bmatrix} 
\end{aligned}
\label{eq:SelfDualPath}
\end{equation}
which are self-dual and Hermitian.
Here the plotting looks a lot better, shown in Figure~\ref{fig:Sphere_to_lemniscate_Pfaffian}.
Also, we can calculate a $\mathbb{Z}_{2}$ invariant, the sign of
the archetypal polynomial. Again, this is known to be trivial ($+1$)
far from the origin, and so a value of $-1$ of the invariant disallows
finite cardinality of the Clifford spectrum.
\end{example}

\subsection{An index for even and odd matrices}

Moving up a dimension, consider
\begin{equation}
\label{eqn:even_odd}
\begin{aligned}
X & =\begin{bmatrix}
0 & 2 & 0 & 0\\
2 & 0 & 0 & 0\\
0 & 0 & 0 & -2\\
0 & 0 & -2 & 0
\end{bmatrix}
, \quad
 Y =\begin{bmatrix}
0 & -i & 0 & 0\\
i & 0 & 0 & 0\\
0 & 0 & 0 & -i\\
0 & 0 & -i & 0
\end{bmatrix},\\
Z & =\begin{bmatrix}
1 & 0 & 0 & 0\\
0 & -1 & 0 & 0\\
0 & 0 & -1 & 0\\
0 & 0 & 0 & 1
\end{bmatrix}
, \quad
H =\begin{bmatrix}
0 & 0 & 1 & 0\\
0 & 0 & 0 & 1\\
1 & 0 & 0 & 0\\
0 & 1 & 0 & 0
\end{bmatrix}.
\end{aligned}
\end{equation}
The characteristic polynomial of these four matrices, computed by
the code in the supplementary file \texttt{Even\_odd\_4CMathematica.nb}, 
is
\[
\left(R^{4}+2R^{2}w^{2}+6R^{2}+w^{4}-6w^{2}+9\right)\left(R^{4}+2R^{2}w^{2}+14R^{2}+w^{4}+2w^{2}-15\right)
\]
where
$ R^{2}=x^{2}+y^{2}+z^{2}$.  
Again we have a surface homeomorphic to a three-sphere.

We introduce a grading via the matrix
\begin{equation*}
\Gamma=\begin{bmatrix}
1 & 0 & 0 & 0\\
0 & 1 & 0 & 0\\
0 & 0 & -1 & 0\\
0 & 0 & 0 & -1
\end{bmatrix}
\end{equation*}
so we consider a matrix $M$ \emph{even} if $M\Gamma=\Gamma M$ and
\emph{odd} if $M\Gamma=-\Gamma M$.
In the example under discussion, the first three matrices are
even and the last is odd.

With these symmetries, we get an index for points
$(w,x,y,z)$ not in the Clifford spectrum \emph{and with the restriction
that $z=0$}. This restriction is needed as translating $H$ will
ruin the symmetry $H\Gamma=\Gamma H$.  The index is based on the fact
that
\begin{equation*}
i\widetilde{L}_{\boldsymbol{\lambda}}(X,Y,Z,Y)\left(\Gamma\otimes I_{2}\right)
\end{equation*}
is Hermitian, and the index is
\begin{equation*}
\frac{1}{2}\mathrm{Sig}\left(i\widetilde{L}_{\boldsymbol{\lambda}}(X,Y,Z,Y)\left(\Gamma\otimes I_{2}\right)\right).
\end{equation*}
Here we are referring the the reduced localizer of (\ref{eqn:reduced_localizer}).
This is explained in \cite{LoringSB_odd_dim}.

\begin{figure}
\includegraphics{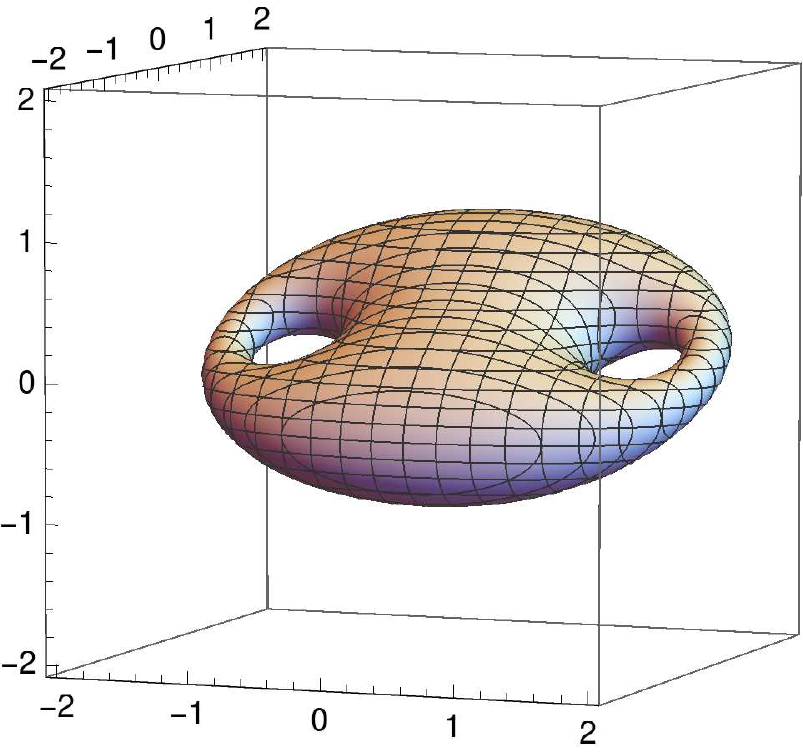}
\includegraphics{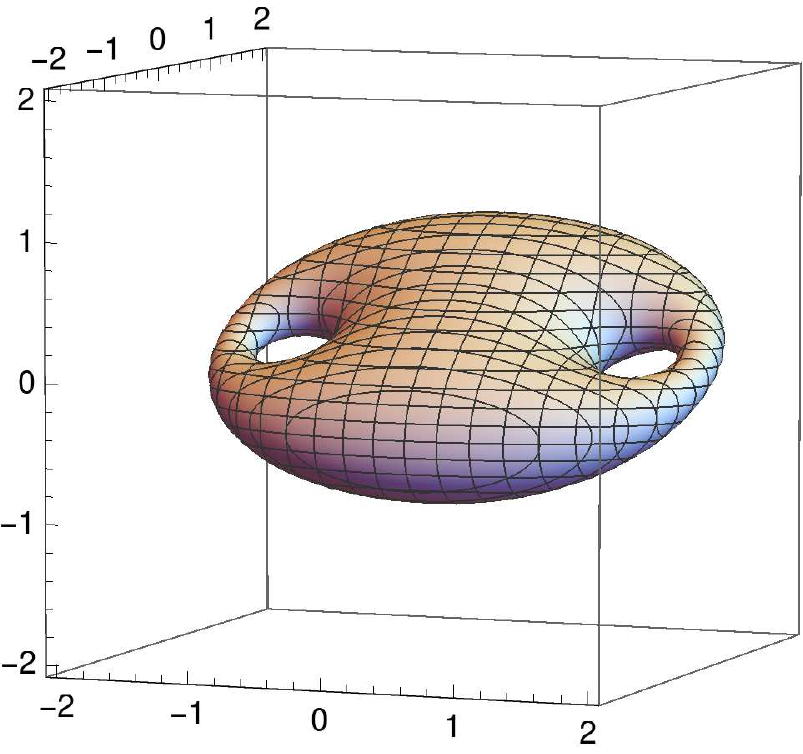}\\
\includegraphics{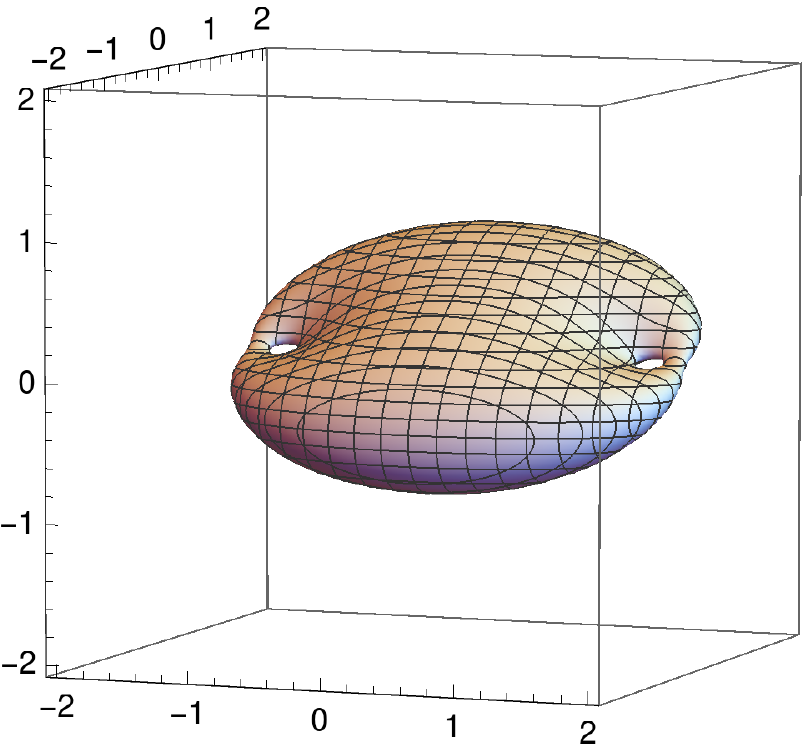}
\includegraphics{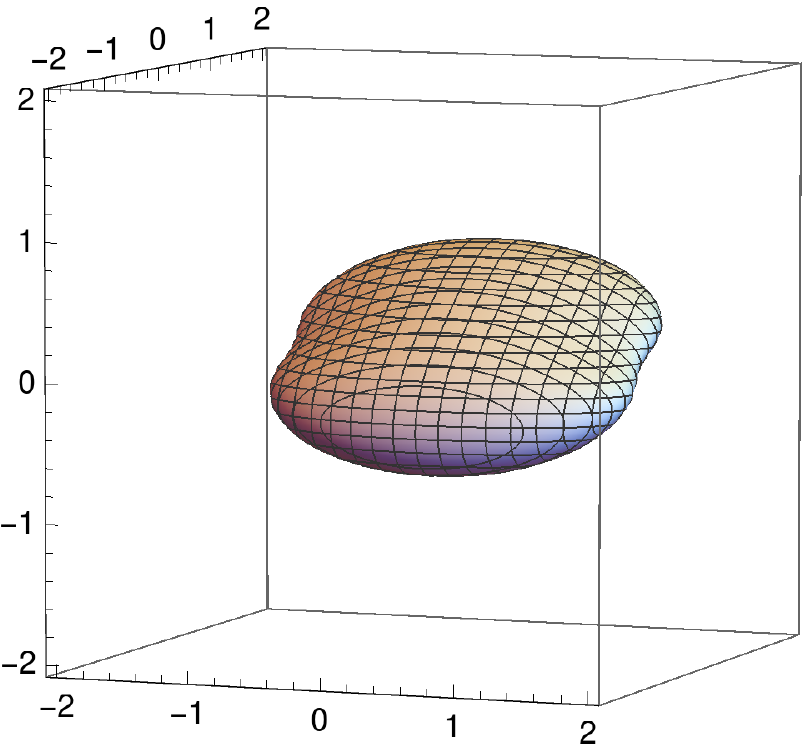}
\caption{This figure shows parts of the Clifford spectrum for 
the four matrices in (\ref{eq:deformed_3_sphere_matrices}).  Shown are slices of the Clifford spectrum in 4-space
through the hyperplances
$z=0,$ $z=0.2$, $z=0.4$ and $z=0.6$. There is a $\pm z$ symmetry in this example so
these images are valid for the corresponding negative values of $z$.
This plots we created using the file \texttt{Even\_odd\_4A.nb}.
\label{fig:Holes_in_hypersurface}}
\end{figure}

For the matrices in (\ref{eqn:even_odd}),
the index at the origin is $-1$.  As always for lambda large compared to the
norm of the matrices the index is $0$. Thus the part of the Clifford spectrum that intersects the hyperplane
$z=0$ is protected. Small perturbations of the matrices will not
change by much the part of the Clifford spectrum intersected with $z=0$.

A little exploration of matrices near these lead
to the following.  Consider the four matrices 
\begin{equation}
 \label{eq:deformed_3_sphere_matrices}
\begin{aligned}
X & =\begin{bmatrix}
\tfrac{3}{2} & 2 & 0 & 0\\
2 & 0 & 0 & 0\\
0 & 0 & 0 & -2\\
0 & 0 & -2 & \tfrac{3}{2}
\end{bmatrix}
,\quad 
Y =\begin{bmatrix}
0 & -i & 0 & 0\\
i & 0 & 0 & 0\\
0 & 0 & 0 & -i\\
0 & 0 & -i & 0
\end{bmatrix},                     \\
Z & =\begin{bmatrix}
1 & 0 & 0 & 0\\
0 & -1 & 0 & 0\\
0 & 0 & -1 & 0\\
0 & 0 & 0 & 1
\end{bmatrix}
, \quad 
H =\begin{bmatrix}
0 & 0 & 1 & 0\\
0 & 0 & 0 & 1\\
1 & 0 & 0 & 0\\
0 & 1 & 0 & 0
\end{bmatrix},
\end{aligned}
\end{equation}
so $r=0$ recreates the previous example. Figure~\ref{fig:Holes_in_hypersurface} 
looks at slices of the Clifford spectrum for this example.

\section*{Supplemetary files}
The supplementary files are available for download from 
\begin{quote}
\url{math.unm.edu/~loring/CliffordExperiments/} 
\end{quote}
and are all Mathematica files, videos created Mathematica files or
a PDF copy of a Mathematica file.

\section*{Acknowledgments}

The research of all authors for this project was supported in part by the National Science Foundation (DMS \#1700102).

\bibliographystyle{amsplain}
\bibliography{CiffordRefs}


\end{document}